\documentclass[reqno,11pt]{amsart}
\usepackage{amsmath,amssymb,mathrsfs,amsthm,amsfonts} 
\usepackage[usenames,dvipsnames]{xcolor}
\usepackage{hyperref}
\usepackage{comment}
\usepackage{stmaryrd}
\usepackage{pgfplots}
\usetikzlibrary{arrows}
\hypersetup{%
	colorlinks=true, linkcolor=blue,
	citecolor=ForestGreen
}
\usepackage[paper=letterpaper,margin=1in]{geometry}
\usepackage{graphicx,tikz}
\usepackage{bm}
\usepackage{marginnote}
\usepackage{bbm}
\usepackage{caption}
\usepackage{tikz}
\usepackage{enumerate}
\usetikzlibrary{matrix,graphs,arrows,positioning, shapes, calc,decorations.markings,decorations.pathmorphing,shapes.symbols,hobby}
\usepackage{dsfont}
\usepackage{todonotes}

\newtheorem{theorem}{Theorem}[section]
\newtheorem{lemma}[theorem]{Lemma}

\newtheorem{proposition}[theorem]{Proposition}
\newtheorem{corollary}[theorem]{Corollary}
\theoremstyle{definition}
\newtheorem{definition}[theorem]{Definition}

\numberwithin{example}{section}
\newtheorem{remark}[theorem]{Remark}
\numberwithin{equation}{section}
\allowdisplaybreaks
\usepackage{acronym}

\newcommand{\cc}{CC}
\newcommand{\cnc}{CNC}

\newcommand{\E}{\mathbb{E}}
\newcommand{\N}{\mathbb{N}}
\renewcommand{\P}{\mathbb{P}}
\renewcommand{\Pr}{\mathbb{P}}	
\newcommand{\Ex}{\mathbb{E}}	
\renewcommand{\d}{\mathrm{d}}	
\newcommand{\ind}{\mathds{1}}	

\newcommand{\cM}{\mathcal{M}}

\newcommand{\Con}{\mathrm{C}} 

\newcommand{\R}{\mathbb{R}} 

\DeclareMathOperator\supp{supp}

\DeclareMathOperator\PT{Pt}
\DeclareMathOperator\Em{Em}

\newcommand{\til}{\widetilde}
\renewcommand{\bar}{\overline}
\renewcommand{\hat}{\widehat}

\usepackage{graphicx}

\title[Large deviation principle for random permutations]{Large deviation principle for random permutations}

\author[J.\ Borga]{Jacopo Borga}
\address{J.\ Borga,
	Department of Mathematics, Stanford University,
	\newline\hphantom{\quad \ \ J. \ Borga}
	450 Jane Stanford Way,
	Stanford, CA 94305, United States of America
}
\email{jborga@stanford.edu}

\author[S.\ Das]{Sayan Das}
\address{S.\ Das,
	Department of Mathematics, Columbia University,
	\newline\hphantom{\quad \ \ S. Das}
	2990 Broadway, New York, NY 10027, United States of America
}
\email{sayan.das@columbia.edu}

\author[S.\ Mukherjee]{Sumit Mukherjee}
\address{S.\ Mukherjee,
	Department of Statistics, Columbia University,
	\newline\hphantom{\quad \ \ S. Mukherjee}
	1255 Amsterdam Avenue, New York, NY 10027, United States of America.
}
\email{sm3949@columbia.edu}
\author[P.\ Winkler]{Peter Winkler}
\address{P.\ Winkler,
	Department of Mathematics, Dartmouth College,
	\newline\hphantom{\quad \ \ P. Winkler}
	27 N. Main Street, Hanover, NH 03755, United States of America.
}
\email{peter.winkler@dartmouth.edu}

\begin{document}
	\begin{abstract} 

		We derive a large deviation principle for random permutations induced by probability measures of the unit square, called  \emph{permutons}. These permutations are called  $\mu$-\emph{random permutations}. We also introduce and study a new general class of models of random permutations, called \emph{Gibbs permutation models}, which combines and generalizes $\mu$-random permutations and the celebrated \emph{Mallows model} for permutations. Most of our results hold in the general setting of Gibbs permutation models.
		
		We apply the tools that we develop to the case of $\mu$-random permutations conditioned to have an atypical proportion of patterns. Several results are made more concrete in the specific case of inversions. For instance, we prove the existence of at least one phase transition for a generalized version of the Mallows model where the base measure is non-uniform. This is in contrast with the results of Starr (2009, 2018) on the (standard) Mallows model, where the absence of phase transition, i.e., phase uniqueness, was proven.
		
		Our results naturally lead us to investigate a new notion of permutons, called \emph{conditionally constant} permutons, which generalizes both \emph{pattern-avoiding} and \emph{pattern-packing} permutons. We describe some properties of conditionally constant permutons with respect to inversions. The study of \emph{conditionally constant} permutons for general patterns seems to be a new challenging problem.
	\end{abstract}
	
	\subjclass[2020]{05A05, 60C05, 60F10}
	\keywords{Random permutations, patterns, permutons, Gibbs measures}

	\maketitle

\section{Introduction}

Studying random permutations is of central importance in probabilistic combinatorics, with the uniform distribution on the symmetric group receiving the most attention (\cite{diaconis1988group}). In the context of large deviations, a large deviation principle (henceforth referred to as LDP) for a uniformly random permutation was first derived in \cite{Trashorras}. More recently,  \cite{Mukherjee} and \cite{Winkler}  gave independent proofs of this important result, and use it to study behavior of exponential tilts, and uniformly random permutations under rare events. 

In this paper we study an LDP for $\mu$-random permutations, where $\mu$ is a probability distribution on $[0,1]^2$ with continuous marginals, {i.e.\ marginals with continuous cumulative distribution functions.} The notion of $\mu$-random permutations was introduced in \cite{Hoppen}, and we recall it in the next definition. 

We denote the set of all permutations of size $n$ by $S_n$ and the infinite set of all permutations of finite size by $S$.

\begin{definition}\label{def:mu_random}

Given $n$ points $({\bf x},{\bf y})=\big((x_i,y_i)\big)_{1\le i\le n} \in [0,1]^2$, with $x_i\neq x_j$ and $y_i\neq y_j$ for all $i\neq j$, we define a permutation $\pi_{{\bf x},{\bf y}}\in S_n$ associated to it as follows. 
Let $(x_{(i)},y_{(i)})_{1\le i\le n}$ be the \emph{$x$-reordering} of  $(\bf{x},\bf{y})$, i.e.\ the unique reordering of the sequence $\big((x_i,y_i)\big)_{1\le i\le n}$ such that
$x_{(1)}<\cdots<x_{(n)}$.
The values $(y_{(1)},\ldots,y_{(n)})$ are then in the same
relative order as the values of a unique permutation $\pi_{{\bf x},{\bf y}}\in S_n$, called the  \emph{permutation induced by} $(\bf{x},\bf{y})$. 

Let $\mathcal{M}$ denote the set of all probability measures on $[0,1]^2$ with continuous marginals. For any $\mu\in \cM$ and $n\in\N$,  define the $\mu$-\emph{random permutation} $\pi_{n,\mu}$ of size $n$ as 
\begin{equation}\label{eq:muranper}
	\pi_{n,\mu}:=\pi_{{\bf X},{\bf Y}}
\end{equation}
where $({\bf X},{\bf Y})=((X_1,Y_1),\ldots,(X_n,Y_n))\stackrel{i.i.d.}{\sim}\mu$. Since $\mu$ has continuous marginals, there are no ties between $(X_1,\cdots,X_n)$ and $(Y_1,\cdots,Y_n)$ with probability $1$. It follows that $\pi_{n,\mu}$ is almost surely well defined. 
\end{definition}

\begin{remark}
	In \cite{Hoppen} the authors assume that $\mu$ belongs to $\widetilde{\cM}$, the set of measures $\mu\in \cM$ with uniform marginals. In this case we call $\mu$ a \emph{permuton}, as first done in \cite{glebov2015finitely}. We point out that for the purpose of defining $\pi_{n,\mu}$, this assumption can be made without loss of generality. To see this, let $F_1,F_2$ denote the marginal cumulative distribution functions of $\mu\in \cM$. Then with $(X,Y)\sim \mu$, the random vector $(F_1(X), F_2(Y))$ is supported on the unit square and has uniform marginals. Also {note that the marginal distribution functions $F_1$ and $F_2$ are strictly increasing $\mu^{\otimes 2}$-almost surely, i.e.\ if $(X_1,Y_1), (X_2,Y_2)\stackrel{i.i.d.}{\sim}\mu$, then $X_1<X_2$ implies $F_1(X_1)<F_1(X_2)$ with $\mu^{\otimes 2}$-probability 1, and similarly for $F_2$. Consequently,}  one can see that $\pi_{n,\mu}\stackrel{d}{=}\pi_{n,\tilde{\mu}}$, where $\tilde{\mu}$ is the law of $(F_1(X),F_2(Y))$. We denote by $\mathcal{O}$ the map 
	\begin{align}\label{eq:map_proj}
		\mathcal{O}:\mathcal{M}&\rightarrow \widetilde{\mathcal{M}}\\
		\mu&\mapsto \widetilde{\mu}.\nonumber
	\end{align}
	Then $\mathcal{O}$ is onto, but not 1-1.
	Thus one can usually assume that $\mu$ is supported on $[0,1]^2$, and has uniform marginals. We do not, however, make this assumption, since we will need to consider  probability measures $\mu\in \cM$ with continuous (but possibly non-uniform) marginals (see for instance Remark \ref{rem:rets} below).
\end{remark}

\subsection{Permutation limit theory and permutons} To study large deviation for permutations, it is necessary to embed permutations of all sizes in a common topological space. This was first done in \cite{Hoppen}, where the authors study a notion of permutation limits, motivated by the study of property testing of permutations (\cite{hoppen2}) and the notion of dense graph limits or graphons. In this section we give a brief introduction to the permutation limit theory (see
e.g.\ \cite[Section 2.1]{borga2021random} for more detail).

\begin{definition}
	Let $\mathcal{P}$ denote the space of all probability measures on $[0,1]^2$, and note that $\widetilde{\mathcal{M}}\subset \mathcal{M}\subset
	\mathcal{P}$.  Given a permutation $\pi\in S_n$, the associated \emph{empirical measure} $\Em(\pi)\in \mathcal{P}$ is defined as
	\begin{align}\label{def:emp}
		\Em(\pi):=\frac1n\sum_{i=1}^n \delta_{\left(\frac{i}{n},\frac{\pi(i)}{n}\right)},
	\end{align}
	where $\delta_{x}$ denotes the Dirac probability measure at $x\in[0,1]^2$.
\end{definition}

Given a pattern $\sigma\in S_k$ and a probability measure $\nu\in \mathcal{P}$, denote the \emph{pattern density} of $\sigma$ in the measure $\nu$ by setting
\begin{align}\label{def:ts}
	t_\sigma(\nu):=\nu^{\otimes k}(h_\sigma)=\int_{[0,1]^{2k}} h_\sigma\Big[(x_1,y_1),\cdots,(x_k,y_k)\Big] d\nu(x_1,y_1)\cdots d\nu(x_k,y_k),
\end{align}
where
\begin{align}\label{def:hs1}
	\notag &h_\sigma\big((x_1,y_1),\cdots,(x_k,y_k)\big):=\mathds{1}\{\pi_{(x_1,\ldots,x_k),(y_1,\ldots,y_k)}=\sigma\}\\
	=
&\sum_{\tau\in S_k} \mathds{1}\left\{ x_{\tau(1)}<x_{\tau(2)}<\ldots<x_{\tau(k)}\right\}   \mathds{1}\left\{y_{\tau(\sigma^{-1}(1))}<y_{\tau(\sigma^{-1}(2))}<\ldots<y_{\tau(\sigma^{-1}(k))}\right\}.
\end{align}
The combinatorial and probabilistic properties of pattern densities are well studied in the literature (see \cite{goldstein2005berry,bona,janson,janson2020patterns,borga2021asymptotic} and the references therein). Note that $t_\sigma(\nu)\in [0,1]$, for all $\sigma\in S_k$ and all $\nu\in \mathcal{P}$. We can naturally extend the mapping $t_\sigma(\cdot)$ to permutations (instead of probability measures) by setting  for a permutation $\pi\in S_n$ and a pattern $\sigma\in S_k$, 
\begin{equation}\label{eq:extend_map}
	t_\sigma(\pi):=t_\sigma(\Em(\pi)).
\end{equation}

Suppose $\pi_n$ is a sequence of permutations with $\pi_n\in S_n$. We say that the sequence $(\pi_n)_{n\ge 1}$ \emph{converges} if for every pattern $\sigma\in S$, the pattern density $t_\sigma(\pi_n)$ converges. In this case, it was shown in \cite{Hoppen} that there exists a permuton $\nu_\infty\in \widetilde{\cM}$, such that 
\begin{equation}\label{eq:perm_lim}
	\lim_{n\to\infty}t_\sigma(\pi_n)=t_\sigma(\nu_\infty).
\end{equation}
We then say that $\pi_n$ \emph{converges to} $\nu_\infty$.
Moreover, given any $\nu \in \widetilde{\cM}$, there exists a sequence of permutations $\pi_n\in S_n$ which converges to $\nu$. Thus $\widetilde{\cM}$ is exactly the set of all possible limits of permutations. 

Another important result of \cite{Hoppen} is that a sequence of permutations $(\pi_n)_{n\ge 1}$ with $\pi_n\in S_n$ converges to a measure $\nu_\infty\in \widetilde{\cM}$ if and only if the sequence of probability measures $\Em(\pi_n)$ converges in the weak topology to $\nu_\infty$. It also follows from this result that the space of all permutation limits, equipped with the topology of pattern convergence, is compact.

\begin{remark}
	In the rest of the paper, whenever we say that a sequence of permutations $\pi_n$ converges to a permuton $\nu_\infty$, we mean that the corresponding sequence $\Em(\pi_n)$ converges weakly to $\nu_\infty$. Similarly, if we say that a sequence of random permutations $\pi_n$ satisfies a certain LDP, we mean that the corresponding sequence $\Em(\pi_n)$ satisfies the LDP on $\mathcal P$ with respect to weak topology.
\end{remark}

\begin{remark}
	In the permuton literature (see the discussion after this remark for a brief overview) it is more common to associate a sequence of permutations $\pi_n$ with the sequence of permutons $\nu_{\pi_n}$ (rather than the probability measures $\Em(\pi_n)$) defined  by
	\begin{equation*}
		\nu_{\pi_n}(A)
		= 
		n 
		\sum_{i=1}^n 
		\text{Leb}
		\big([(i-1)/n, i/n]
		\times
		[(\sigma(i)-1)/n,\sigma(i)/n]
		\cap 
		A\big),
	\end{equation*}
	for all Borel measurable sets $A$ of $[0,1]^2$.
	In words, $\nu_{\pi_n}$ can be obtained by uniformly distributing mass to the squares $\{[\frac{i-1}{n}, \frac{i}{n}]\times  [\frac{\sigma(i)-1}{n}, \frac{\sigma(i)}{n}]: 1\leq i \leq n\}.$ We highlight that all our results hold also with this different encoding, applying for instance \cite[Theorem 4.2.13]{DZ}, since the Kolmogorov-Smirnov distance between $\nu_{\pi_n}$ and $\Em(\pi_n)$ is at most $\frac1n$.
\end{remark}

Permuton limits have been investigated for various models of random permutations and there is a growing literature in the past decade. {Permuton limits have been used to study several permutation statistics of interest, such as fixed points, number of cycles of a given length, permutation graphs, and the longest increasing subsequence~\cite{mueller2013length,muk2,muk3,lis-bass-22,borga2022permutons,dubach2023locally}.} 
For many models, the permuton limits are deterministic permutons, for instance, Erd\"{o}s-Szekeres permutations \cite{MR2266895}, Mallows random permutations \cite{Shannon,starr2018phase}, certain classes of exponential families on permutations \cite{Mukherjee}, random sorting networks \cite{dauvergne2018archimedean}, permutations avoiding decreasing sequences \cite{hoffman2017pattern,hoffman2019scaling}, permutations with fixed pattern densities \cite{Winkler}, almost square permutations \cite{borga2021almost}, and permutations sorted with the {\sc{runsort}} algorithm \cite{alon2021runsort}. For random pattern-avoiding permutations, the limiting permutons appear to be random in many cases. In~\cite{borga2021skewperm} a two-parameter family of permutons, called the \emph{skew Brownian permuton}, was introduced to cover most of the known examples \cite{bassino2018separable,bassino2017universal, MR4115736,bassino2019scaling,borga2021permuton}.

{In a different spirit compared to ours, large deviation results for pattern-avoiding permutations has been studied in the literature \cite{miner2014shape,atapour2014large,madras2016large}. Perhaps more related to the current research is the LDP for random graphs with respect to the \emph{cut metric} (see \cite{chatterjee2011large, borgs2020large, dhara2022large} and references there-in). In \cite{chatterjee2011large} the authors study an LDP for  Erd\"{o}s-R\'enyi random graphs. Utilizing this LDP, \cite{chatterjee2013estimating} investigates Exponential Random Graph Models (ERGMs), which are exponential families on the space of graphs. On the other hand, \cite{DL,kenyon2018bipodal} study the behavior of random graphs constrained by subgraph densities. More recently, \cite{dhara2022large} and \cite{borgs2020large} generalize this LDP to the setting of inhomogeneous random graphs. In a similar manner, in the permutation world \cite{Trashorras,Mukherjee,Winkler} establish the LDP for a uniformly random permutation.  Utilizing this, \cite{Mukherjee} studies exponential families on the space of permutations (of which the Mallows models is a special case), and \cite{Winkler} studies conditional behavior of a uniformly random permutation under constraint on pattern densities. The LDP for $\mu$ random permutations (studied here) is analogous to the LDP for inhomogeneous random graphs. As will be explained below, this allows us to study a much more general class of probability measures on permutations, and demonstrate interesting phase transition properties there-in. }


\subsection{Large deviation principle for $\mu$-random permutations}\label{sec1.2}

We begin with our first main result that establishes an LDP for random empirical measures corresponding to $\mu$-random permutations. 

To describe the \emph{good rate function} (see \cite[Section 1.2]{DZ} for basic definitions related to LDP), we need the following notation:
Let $(\sigma_i)_{i\ge 1}$ be an enumeration of the set of all patterns of all sizes. Define a mapping\footnote{Note that the mapping $\PT$ is defined from $\mathcal{P}$ to $[0,1]^\N$, but here it is only used when restricted to $\cM$. The motivation is that later we will also need to use the mapping $\PT$ on the space $\mathcal P$ (see for instance \eqref{eq:gen_map}).} $\PT:\mathcal{P}\to[0,1]^\N$ by setting
\begin{equation}\label{eq:same_prop}
	\PT(\gamma)=\left(t_{\sigma_i}(\gamma)\right)_{i\ge 1}.
\end{equation}
It follows from \cite{Hoppen} that given any $\mu\in \cM$, there exists a unique $\widetilde{\mu}$ in $\widetilde{\cM}$ such that $\PT(\mu)=\PT(\widetilde{\mu})$. In particular, $\widetilde{\mu}=\mathcal{O}(\mu)$, where $\mathcal{O}(\cdot)$ is the mapping introduced in \eqref{eq:map_proj}.

\medskip

The claimed good rate function for the LDP for $\pi_{n,\mu}$ is 
\begin{eqnarray}\label{eq:rate}
	\begin{split}
		I_\mu(\gamma)= \begin{cases}
			\inf_{\nu\in \cM:\mathcal{O}(\nu)=\gamma} D(\nu|\mu) &\text{ if }\gamma\in \widetilde{\cM},\\
			\infty& \text{ if }\gamma\notin  \widetilde{\cM},
		\end{cases}
	\end{split}
\end{eqnarray}
where $D(\nu|\mu)$ denotes the \emph{Kullback–Leibler divergence}, i.e.\ $D(\nu|\mu)=\int_{[0,1]^2}\log(\frac{d\nu}{d\mu})d\nu$.

\medskip

Our first main result shows that $\pi_{n,\mu}$ (introduced in Definition \ref{def:mu_random}) satisfies an LDP with the above good rate function.

\begin{theorem}\label{thm:main} 
	For any probability measure $\mu\in \mathcal{M}$, the random permutations $\pi_{n,\mu}$ satisfy an LDP with speed $n$ and good rate function $I_\mu(\cdot)$ as in \eqref{eq:rate}.  More precisely, for any Borel set $A\subset \mathcal{P}$ we have
	\begin{equation*}
		-\inf_{\gamma\in  A^\circ} I_\mu(\gamma)\le \liminf_{n\to\infty}\frac{1}{n}\log \P(\Em(\pi_{n,\mu})\in A)
		\le \limsup_{n\to\infty}\frac{1}{n}\log \P(\Em(\pi_{n,\mu})\in A)\le -\inf_{\gamma\in \bar{A}} I_\mu(\gamma),
	\end{equation*}
where $A^\circ$ and $\bar{A}$ denote the interior and the closure of $A$, respectively. 
\end{theorem}

We note that when $\mu=\lambda$ is the Lebesgue measure on $[0,1]^2$, the permutation $\pi_{n,\mu}$ is uniformly random on $S_n$, the set of all permutations of size $n$. As noted before, in this case the LDP is well known in the literature.

\begin{corollary}[\cite{Trashorras,Mukherjee,Winkler}]\label{cor:ldp_unif}
	If $\pi_n$ is uniformly random on $S_n$, then it satisfies an LDP on $\mathcal{P}$ with speed $n$ and the good rate function 
	\begin{align}\label{eq:unifgrf}
		I(\gamma):= \begin{cases}
			 D(\gamma|\lambda) & \text{ if }\gamma\in  \widetilde{\mathcal{M}},\\
			 \infty & \text{ otherwise},
		\end{cases}
	\end{align}
	where $\lambda$ is the Lebesgue measure on $[0,1]^2$. 
\end{corollary}

\begin{remark}\label{rem:rets}
	We highlight an important difference between the statements of Theorem \ref{thm:main} and Corollary \ref{cor:ldp_unif}. For a general probability measure $\mu\in \mathcal{M}$, the good rate function $I_\mu(\gamma)$ in \eqref{eq:rate} requires minimizing the divergence $D(\cdot|\mu)$ over the set $\{\nu\in \cM:\mathcal{O}(\nu)=\gamma\}$, that is, it requires considering the larger set $\cM$ in the minimization. In particular, we show in Proposition \ref{ppn:counter_new}  below that one cannot directly restrict the minimization problem over the space of permutons $\widetilde{\mathcal{M}}$, as in the case of uniform random permutations in \eqref{eq:unifgrf}.
\end{remark}

\begin{proposition}\label{ppn:counter_new}
Suppose $X\sim U[0,1]$. If $X<\frac{1}{2}$, set $Y\sim U[0,1/2]$, and if $X>\frac{1}{2}$, set $Y\sim U[1/2,1]$. Let $\mu_1\in \widetilde{\mathcal{M}}$ denote the law of $(X,Y)$. There exists $\gamma\in\til{\mathcal{M}}$ such that the rate function $I_{\mu_1}(\gamma)$ defined in \eqref{eq:rate} equals $D(\nu|\mu_1)<\infty$ for some $\nu \in \mathcal{M}\setminus \til{\mathcal{M}}$, but $D(\gamma|\mu_1)=\infty$. 
\end{proposition}

In the next sections we study several consequences of Theorem \ref{thm:main}.

\subsection{Gibbs random permutations}\label{sec1.3}

As a first application, we introduce and study a class of models on permutations that generalizes both the $\mu$-random permutations introduced above, and the well-studied Mallows model on permutations, in which the probability of a permutation
is proportional to a real parameter raised to the power of
number of inversions of the permutation (see for instance \cite{MR87267,diaconis1988group,Shannon,gladkich2018cycle,diaconis2021statistical,he2021central}).

\begin{definition}\label{def:gibbs}
	
	Fix a pattern $\sigma\in S_k$, $\mu\in \mathcal{M}$, and $\theta\in \R$. Define a \emph{Gibbs probability distribution} $Q_{n,\sigma,\mu,\theta}$ on $[0,1]^{2n}$ by setting
	\begin{align}\label{eq:gibbs_general}
		\frac{dQ_{n,\sigma,\mu,\theta}}{d\mu^{\otimes n}}({\bf x},{\bf y})=\exp\Big(n\theta t_\sigma(\pi_{{\bf x},{\bf y}})
		-F_n(\sigma,\mu,\theta)\Big),
	\end{align}
	where $({\bf x},{\bf y})=((x_1,y_1),\ldots,(x_n,y_n))\in [0,1]^{2n}$ and $F_n(\sigma,\mu,\theta)$ is the \emph{log partition function} of the model, that is
	\begin{equation}\label{eq:gibbs_log_part}
		F_n(\sigma,\mu,\theta):=\log\left(\int_{[0,1]^{2n}}\exp\left(n\theta t_\sigma(\pi_{{\bf x},{\bf y}})\right)d\mu^{\otimes n}({\bf x},{\bf y})\right)\stackrel{\eqref{eq:muranper}}{=}\log\E\exp\left(n\theta t_\sigma(\pi_{n,\mu})\right).
	\end{equation}
	Letting $(\widetilde{\bf X},\widetilde{\bf Y})=((\widetilde{X}_1,\widetilde{Y}_1),\ldots,(\widetilde{X}_n,\widetilde{Y}_n))$ be a random vector with distribution $Q_{n,\sigma,\mu,\theta}$, we set
	\begin{equation}\label{eq:Gibss_def_perm}
		\pi_{n,\sigma,\mu,\theta}:=\pi_{\widetilde{\bf X},\widetilde{\bf Y}}
	\end{equation}
	and we refer to it as \emph{Gibbs random permutation} (of size $n$ with parameters  $\sigma\in S_k$, $\mu\in \mathcal{M}$, and $\theta\in \R$).
\end{definition}

In particular, note that:
\begin{itemize}
	\item if we set $\theta=0$, then $\pi_{n,\sigma,\mu,0}=\pi_{n,\mu}$ is the $\mu$-random permutation introduced before in Definition \ref{def:mu_random};
	\item if  $\mu=\lambda$ is Lebesgue measure, then one can check (a proof is given later in \eqref{eq:gibbs_lebesgue2}) that $\pi_{n,\sigma,\lambda,\theta}$ has the probability mass function
	\begin{align}\label{eq:gibbs_lebesgue}
		\P(\pi_{n,\sigma,\lambda,\theta}=\tau)=\exp\Big(n\theta t_\sigma(\tau)-Z_n(\sigma,\theta)\Big),\quad\text{for all}\quad \tau\in S_n,
	\end{align}
	where $Z_n(\sigma,\theta):=\log(n!)+F_n(\sigma,\lambda,\theta)$;
	\item if\footnote{We use the \emph{one-line notation} to write permutations, that is, if $\sigma\in S_{n}$ then we write $\sigma=\sigma(1)\cdots  \sigma(n)$.} $\sigma=21$ in \eqref{eq:gibbs_lebesgue}, we get the Mallows model on permutations, which has been of significant interest in probability and combinatorics.
\end{itemize}  

\medskip

Our first result studies the typical behavior of Gibbs random permutations. 

\begin{theorem}\label{thm:gibbs_random}
	Fix a pattern $\sigma\in S_k$, $\mu\in \mathcal{M}$ and $\theta\in \R$. Let $\pi_{n,\sigma,\mu,\theta}$ be the Gibbs random permutation introduced in Definition \ref{def:gibbs}. The following statements hold:
	
	\begin{enumerate}[(i)]
		\item  With $F_n(\sigma,\mu,\theta)$ as introduced in \eqref{eq:gibbs_log_part}, we have 
		\begin{align}\label{eq:opt}
			\frac{F_n(\sigma,\mu,\theta)}{n}\to F(\sigma,\mu,\theta):=\sup_{\nu\in \mathcal{M}}\{\theta t_\sigma(\nu)-D(\nu|\mu)\}.
		\end{align}
		
		\item At least one maximizer of the optimization problem in the right-hand side of \eqref{eq:opt} exists. Let $\nu_{\sigma,\mu,\theta}$ be any such maximizer. Then $\nu_{\sigma,\mu,\theta}\ll \mu$, and $g:=\frac{d\nu_{\sigma,\mu,\theta}}{d\mu}$ satisfies the Euler-Lagrange equation
		\begin{equation}
			g(z_1)\stackrel{\mu-\text{a.s.}}{=}\frac{\exp\Big(k\theta \int_{[0,1]^{2k-2}}h_\sigma(z_1,\ldots,z_k) \prod_{a=2}^k g(z_a) d\mu(z_a)\Big)}{\int_{[0,1]^2} \exp\Big(k\theta \int_{[0,1]^{2k-2}}h_\sigma(z_1,\ldots,z_k) \prod_{a=2}^k g(z_a) d\mu(z_a)\Big)},
		\end{equation}
		where $z_i=(x_i,y_i)\in[0,1]^2$.
		
		\item  The random permutations $\pi_{n,\sigma,\mu,\theta}$ satisfy an LDP with speed $n$ and good rate function
		\begin{align*}
			I_{\sigma,\mu,\theta}(\gamma):=\begin{cases}
				\inf_{\nu\in \cM:\mathcal{O}(\nu)=\gamma} \{D(\nu|\mu)-\theta t_\sigma(\nu)\}-\inf_{\nu \in \cM} \{D(\nu|\mu)-\theta t_\sigma(\nu)\} & \text{ if }\gamma\in \widetilde{\cM},\\
				\infty & \text{ otherwise}.
			\end{cases}
		\end{align*}
		
		\item  Let $\mathcal{F}(\sigma,\mu,\theta)$ denote the set of optimizers of part $(i)$. Then 
		\begin{equation*}
			d\Big(\pi_{n,\sigma,\mu,\theta},\mathcal{O}(\mathcal{F}(\sigma,\mu,\theta))\Big)\xrightarrow{P}0,
		\end{equation*}
		where $d(\cdot,\cdot)$ is any metric which characterizes weak convergence and $\mathcal{O}$ is the map introduced in \eqref{eq:map_proj}.
	\end{enumerate}
\end{theorem}

\begin{remark}
	We highlight that it is necessary to apply the mapping $\mathcal{O}$ in part $(iv)$. Indeed, as already pointed out below \eqref{eq:perm_lim}, limits of random permutations have uniform marginals and so live in the space of permutons $\widetilde{\mathcal{M}}$.
\end{remark}

Thus finding out the set of optimizers $\mathcal{F}(\sigma,\mu,\theta)$ is of interest, as they characterize the limits of Gibbs random permutations. The following theorem studies Gibbs random permutations in the so called high-temperature phase (borrowing a term from statistical physics terminology), and shows that in this case there is always a unique maximizer, which behaves nicely under perturbations.

\begin{theorem}\label{thm:gibbs_random2}
	Fix a pattern $\sigma\in S_k$, $\mu\in \mathcal{M}$ and $\theta\in \R$. Let $\pi_{n,\sigma,\mu,\theta}$ be a Gibbs random permutation introduced in Definition \ref{def:gibbs}. Then there exists $\theta_c>0$ (depending only on $k$, i.e.\ the size of the pattern $\sigma$), such that for $\theta\in (-\theta_c,\theta_c)$ the following hold:
	
	\begin{enumerate}[(i)]
		\item The optimization problem in the right-hand side of \eqref{eq:opt} has a unique solution $\nu_{\sigma,\mu,\theta}$, say. Further,  $$\pi_{n,\sigma,\mu,\theta}\xrightarrow{P}\mathcal{O}(\nu_{\sigma,\mu,\theta}).$$
		
		\item The function $\theta\mapsto t_\sigma(\nu_{\sigma,\mu,\theta})$ is continuous and non-decreasing. Further, the map $\theta\mapsto F(\sigma,\mu,\theta)$ is differentiable, with $F'(\sigma,\mu,\theta)=t_\sigma(\nu_{\sigma,\mu,\theta})$.
		
		\item  If $t_\sigma(\nu_{\sigma,\mu,\theta_1})=t_\sigma(\nu_{\sigma,\mu,\theta_2})$, then $\nu_{\theta_1}=\nu_{\theta_2}$.

		\item The map $\theta\mapsto \nu_{\sigma,\mu,\theta}$ from $(-\theta_c,\theta_c)$ to $\mathcal{M}$ is continuous in total variation.

		\item The map $\mu\mapsto \nu_{\sigma,\mu,\theta}$ from $\mathcal{M}$ to $\mathcal{M}$ is continuous, where the metric is total variation on the left-hand side, and weak convergence on the right-hand side
	\end{enumerate}
\end{theorem}

\begin{remark}\label{rem:ass_is_needed_1}
	It was shown in \cite{Shannon,starr2018phase} that if $\sigma=21$ and $\mu=\lambda$, then $\mathcal{F}(21,\lambda,\theta)=\mathcal{O}(\mathcal{F}(21,\lambda,\theta))=\{\nu_\theta\}$ is a singleton for all $\theta\in \R$, and consequently, 
	\begin{equation*}
		\pi_{n,21,\lambda,\theta}\xrightarrow{P}\nu_\theta.
	\end{equation*}
	Note that in Theorem \ref{thm:gibbs_random2}, we generalize this result to Gibbs random permutations, allowing for general patterns $\sigma$ and general base measures $\mu$, but only in the regime $|\theta|<\theta_c$. We will actually show later that the assumption $|\theta|<\theta_c$ is needed in order to guarantee the uniqueness of the optimizer in this general setting of Gibbs random permutations. Indeed, in Proposition \ref{prop:counter} below, we will exhibit a permuton $\xi$ such that both sets of optimizers $\mathcal{F}(21,\xi,\theta)$ and $\mathcal{O}(\mathcal{F}(21,\xi,\theta))$ 
	appearing in Theorem \ref{thm:gibbs_random} part $(iv)$ have cardinality 2 for all $\theta>1$.
\end{remark}

For a general pattern $\sigma$, taking $\mu=\lambda$, i.e.\ the Lebesgue measure on $[0,1]^2$, we get the following corollary.

\begin{corollary}\label{cor:exp}
	Suppose $\pi_{n,\sigma,\lambda,\theta}$ is a random permutation with a p.m.f.\ as in \eqref{eq:gibbs_lebesgue} (or equivalently, as introduced in Definition \ref{def:gibbs}).  Then the following conclusions hold:
	\begin{enumerate}[(i)]	
		\item The random permutation $\pi_{n,\sigma,\lambda,\theta}$ satisfies an LDP with speed $n$ and the good rate function 
		\begin{align}\label{eq:gd_rate}
			I_{\sigma,\theta}(\gamma)=\begin{cases}
				D(\gamma|\lambda)-\theta t_{\sigma}(\gamma)-\inf_{\nu\in \widetilde{\cM}}\{D(\nu|\lambda)-\theta t_\sigma(\nu)\} & \text{ if } \gamma\in\widetilde{\mathcal{M}},\\
				\infty & \text{ otherwise}.
			\end{cases}
		\end{align}
		\item Recall the definition of $Z_n(\sigma,\theta)$ from \eqref{eq:gibbs_lebesgue}. For every $\theta\in \R$ we have 
		\begin{equation}\label{eq:opt_unif}
			\frac{Z_n(\sigma,\theta)-\log n!}{n}\to \sup_{\nu\in \til\cM} \left\{\theta t_{\sigma}(\nu)-D(\nu|\lambda)\right\}.
		\end{equation}
		
		\item With $\theta_c>0$ as in Theorem \ref{thm:gibbs_random2}, for all $\theta\in (-\theta_c,\theta_c)$ the optimization problem in the right-hand side of \eqref{eq:opt_unif} has a unique solution $\nu_{\sigma,\lambda,\theta}\in\widetilde{\cM}$. Further,  
		$$\pi_{n,\sigma,\lambda,\theta}\xrightarrow{P}\nu_{\sigma,\lambda,\theta}.$$
	\end{enumerate}
\end{corollary}

As already mentioned, uniqueness of optimizer is not true for general Gibbs random permutations for all $\theta\in\R$, even for the special case of inversions, i.e.\ $\sigma=21$, as we will show below in proposition \ref{prop:counter}.
In this Proposition, we will actually see that the map $\theta\mapsto \nu_{21,\mu,\theta}$ is constant in the interval $(-\infty,1)$. As it turns out, understanding the behavior of the function $\theta\mapsto \nu_{\sigma,\mu,\theta}$ near the origin is an important step in understanding $\mu$-random permutations under constraints. The following definition provides a sufficient condition on $\mu$ under which the map $\theta\mapsto \nu_{\sigma,\mu,\theta}$ is not constant in a small interval around the origin, i.e.\ for $\theta\in(-\theta_c,\theta_c)$. We first introduce a key notion.
\begin{definition}\label{def:cc_cnc}
	Fix $\sigma\in S_k$, and $\mu\in \mathcal{M}$. Let $({\bf X},{\bf Y})=\big((X_i,Y_i)\big)_{1\le i\le k}\stackrel{i.i.d.}{\sim}\mu$. We say that $\mu$ is \emph{conditionally constant} (\cc) with respect to $\sigma$, if there exists $c\in[0,1]$ such that
	\begin{equation*}
		\P(\pi_{{\bf X},{\bf Y}}=\sigma|(X_1,Y_1))=c\qquad \text{a.s.},
	\end{equation*}
	that is, if the random variable
	$\P(\pi_{{\bf X},{\bf Y}}=\sigma|(X_1,Y_1))$ is constant almost surely. Here $\pi_{{\bf X},{\bf Y}}$ is as in Definition \ref{def:mu_random}. If $\mu$ is not CC,  we will call the measure $\mu$ to be \emph{conditionally not constant} (\cnc) with respect to $\sigma$.
\end{definition}

With this definition, our next result shows that for \cnc\ measures $\mu$, the map $\theta\mapsto \nu_{\sigma,\mu,\theta}$ is indeed non-constant in a small interval around the origin.

\begin{proposition}\label{prop:cnc}
	Fix a pattern $\sigma\in S_k$. Suppose $\mu\in\cM$ is \cnc\, and let $\theta_c$ be as in Theorem \ref{thm:gibbs_random2}. Recall also that for all $\theta\in (-\theta_c,\theta_c)$, $\nu_{\sigma,\mu,\theta}$ denotes the unique maximizer of the optimization problem in the right-hand side of \eqref{eq:opt}. Then the following conclusions hold for all $\theta\in (-\theta_c,\theta_c)$:
	
	\begin{enumerate}[(i)]
		\item  $\nu_{\sigma,\mu,\theta}\ne \mu$ for all $\theta\ne 0$.
		
		\item $t_\sigma(\nu_{\sigma,\mu,\theta})>t_\sigma(\mu)$ for $\theta>0$, and the reverse strict inequality holds for $\theta<0$.
	\end{enumerate}
\end{proposition}

\begin{remark}\label{rem:genpa}
	It follows immediately from part $(ii)$ in the proposition above that if $\mu\in \mathcal{M}$ satisfies
	\begin{equation*}
		\P_\mu(\pi_{{\bf X},{\bf Y}}=\sigma)=0,
	\end{equation*}
	i.e.\ $\mu$ is a \emph{pattern avoiding measure}, or
	\begin{equation*}
		\P_\mu(\pi_{{\bf X},{\bf Y}}=\sigma)=\sup_{\gamma\in \mathcal{M}}\{t_{\sigma}(\gamma)\},
	\end{equation*}
	i.e.\ $\mu$ is a \emph{pattern packing measure}, then $\mu$
	must be \cc.  
\end{remark}

Pattern packing/pattern avoiding permutations/measures have an extensive history in the combinatorics literature (see e.g.\ \cite{albert2002packing,presutti2010packing,kitaev2011patterns} or \cite[Chapter 12]{bona2015handbook} and references therein). In particular it is known that pattern avoiding permutations are not unique but their number grows at most exponentially \cite{marcus2004excluded}.  The issue of uniqueness of pattern \emph{packing} permutations is unresolved except for a few patterns of size $\le 4$. 

\medskip

We will later prove in Section \ref{sect:1.5} that the reverse conclusion (compared to Remark \ref{rem:genpa}) is not true in general; indeed, there exist \cc\ measures $\mu$ that are neither pattern packing nor pattern avoiding for $21$. In Section \ref{sect:1.5} we will also give some characterizations of \cc\ measures for $\sigma=21$. The problem of characterizing \cc\ measures for general patterns seems to be an interesting and challenging problem (see also Section \ref{openprob}).

\subsection{$\mu$-random permutations with an atypical proportion of patterns}

Using the above results for Gibbs random permutations, we are able to study the behavior of a $\mu$-random permutation $\pi_{n,\mu}$ conditioned on having an atypical proportion of patterns $\sigma$, i.e.\ an atypical value of $t_\sigma(\pi_{n,\mu})$. Stating the result requires the following definition:

\begin{definition}\label{def:f}
	For all patterns $\sigma\in S$, probability measures $\mu\in \mathcal{M}$ and positive real $\delta\in\R_{\geq0}$, set 
	\begin{align}\label{eq:f}
		G(\sigma,\mu,\delta):=\inf_{\nu\in \mathcal{M}: t_\sigma(\nu)\ge \delta} D(\nu|\mu)\qquad\text{ and }\qquad \mathcal{G}(\sigma,\mu,\delta):=\arg\inf_{\nu\in \mathcal{M}: t_\sigma(\nu)\ge \delta} D(\nu|\mu).
	\end{align}
	Note that the above definition makes sense only if there exists at least one $\nu\in \mathcal{M}$ such that $t_\sigma(\nu)\ge \delta$. Clearly, this depends on the triplet $(\sigma,\mu,\delta)$. To avoid trivial degeneracies, we define 
	$$\alpha_{\sigma}(\mu)=\inf \{\delta\mid G(\sigma,\mu,\delta)=\infty\},$$
	and consider $\delta \in (t_{\sigma}(\mu),\alpha_{\sigma}(\mu))$. Also, note that $G(\sigma,\mu,\cdot)$ is left continuous by lower semi-continuity of $D(\cdot|\mu)$. 
\end{definition}

\begin{theorem}\label{thm:constrain}
	Let $\mu\in \mathcal{M}$ and suppose $\pi_{n,\mu}$ is a $\mu$-random permutation (as in Definition \ref{def:mu_random}). Let $\sigma\in S_k$ be a fixed pattern. Then the following conclusions hold:
	
	\begin{enumerate}[(i)]
		\item Suppose $G(\sigma,\mu,\cdot)$ defined in \eqref{eq:f} is right continuous at $\delta$, for some $\delta>t_{\sigma}(\mu)$. Conditioned on the event $ \{t_\sigma(\pi_{n,\mu})\ge \delta\}$, we have
		\[
		d\Big(\pi_{n,\mu},\mathcal{O}(\mathcal{G}(\sigma,\mu,\delta))\Big)\xrightarrow{P}0,
		\]
		where $d(\cdot,\cdot)$ is any metric which characterizes permutation convergence and $\mathcal{O}$ is the map introduced in \eqref{eq:map_proj}.

		\item Suppose $\mu$ is \cnc\ with respect to $\sigma$. Then, setting 
		$$\delta_c:=\sup_{\theta\in (-\theta_c,\theta_c)} t_{\sigma}(\nu_{\sigma,\mu,\theta})=\lim_{\theta\to \theta_c^-}t_{\sigma}(\nu_{\sigma,\mu,\theta}),$$ 
		(where $\theta_c$ and $\nu_{\sigma,\mu,\theta}$ are as in Theorem \ref{thm:gibbs_random2}), we have $\delta_c>t_\sigma(\mu)$, and the following conclusions hold for $\delta\in (t_\sigma(\mu),\delta_c)$:
		
		\begin{enumerate}[(a)]
			
			\item The number 
			$$\hat{\theta}(\delta):=\max\{\theta>0:t_\sigma(\nu_{\sigma,\mu,\theta})=\delta\}\in (0,\theta_c)$$
			is well defined, and satisfies $t_\sigma\left(\nu_{\sigma,\mu,\hat{\theta}(\delta)}\right)=\delta$.
			
			\item The set $\mathcal{G}(\sigma,\mu,\delta)$ has the single element $\nu_{\sigma,\mu,\hat{\theta}(\delta)}$.

			\item The measure $\nu_{\sigma,\mu,\hat{\theta}(\delta)}$ of part $(b)$ is absolutely continuous with respect to $\mu$, and the Radon-Nikodym derivative $g:=\frac{d\nu_{\sigma,\mu,\hat{\theta}(\delta)}}{d\mu}$ satisfies the Euler-Lagrange equation 
			\[g(z_1)\stackrel{\mu-\text{a.s.}}{=}\frac{\exp\Big(k\hat{\theta}(\delta) \int_{[0,1]^{2k-2}} h_\sigma(z_1,\ldots,z_k) \prod_{a=2}^k g(z_a) d\mu(z_a)\Big)}{\int_{[0,1]^2} \exp\Big(k\hat{\theta}(\delta) \int_{[0,1]^{2k-2}} h_\sigma(z_1,\ldots,z_k) \prod_{a=2}^k g(z_a) d\mu(z_a)\Big)d\mu(z_1)}.\]
			
			\item $G(\sigma,\mu,\cdot)$ is continuous on $(t_{\sigma}(\mu),\delta_c)$. Further, conditioned on the event $\left\{ t_\sigma(\pi_{n,\mu})\ge \delta\right\}$, we have \[\pi_{n,\mu}\xrightarrow{P}\mathcal{O}\left(\nu_{\sigma,\mu,\hat{\theta}(\delta)}\right).\]
		\end{enumerate}
	\end{enumerate}
	
\end{theorem}

\begin{remark}\label{rem:o_directict}
	Theorem \ref{thm:constrain} is stated for upper tail conditioning, i.e.\ $\left\{ t_\sigma(\pi_{n,\mu})\ge \delta\right\}$, but all the results hold also for lower tail conditioning, i.e., $\left\{ t_\sigma(\pi_{n,\mu})\le \delta\right\}$, with the obvious adaptation. 
\end{remark}

\begin{remark}\label{rem:ass_is_needed_2}
	The above theorem requires the \cnc\ assumption for part $(ii)$ to hold, as demonstrated in the counterexample given in Proposition \ref{prop:xi}. There, a permuton $\xi$ is exhibited that is \cc\ for $\sigma=21$ and such that the set $\mathcal{O}(\mathcal{G}(21, \xi,\delta))$ has cardinality $2$ for all $\delta \in (t_{21}(\xi),\alpha_{21}(\xi))$.
\end{remark}

\subsection{Inversions: some concrete examples}\label{sect:1.5}

In this section we focus on the specific case when $\sigma$ is an inversion, i.e., $\sigma=21$. 

\subsubsection{Non-uniqueness of the optimizers}

We start by discussing the non-uniqueness of the optimizer in Theorem \ref{thm:gibbs_random2} (resp.\ Theorem \ref{thm:constrain}) in the absence of the assumption $|\theta|<\theta_c$ (resp.\ $\mu$ is \cnc). Recall also Remarks \ref{rem:ass_is_needed_1} and \ref{rem:ass_is_needed_2}. We first need the following definition.

\begin{definition}\label{defn:counter}
	Let $\xi\in \widetilde{\cM}$ be a permuton defined as follows. Suppose $X\sim U[0,1]$. If $X<\frac{1}{2}$, set $Y=\frac{1}{2}-X$, and if $X>\frac{1}{2}$, set $Y=\frac{3}{2}-X$. Then $(X,Y)$ is a random vector on $[0,1]^2$, and has uniform marginals. Let $\xi$ denote the law of $(X,Y)$. 
	The support of $\xi$ is shown in Figure~\ref{fig:counter}.
	
	Similarly, we set $\xi_{11},\xi_{22}\in \cM$ to be the uniform probability measures on the diagonals of the boxes $D_{11}=[0,\frac12]^2$ and $D_{22}=[\frac12,1]^2$, respectively. Note that $\xi_{11}$ and $\xi_{22}$ are not permutons and that $\xi=\frac12 \xi_{11} +\frac 12 \xi_{22}$.
\end{definition}

\begin{figure}[ht]
	\begin{center}
		\includegraphics[width=2.5in, height=2.5in]{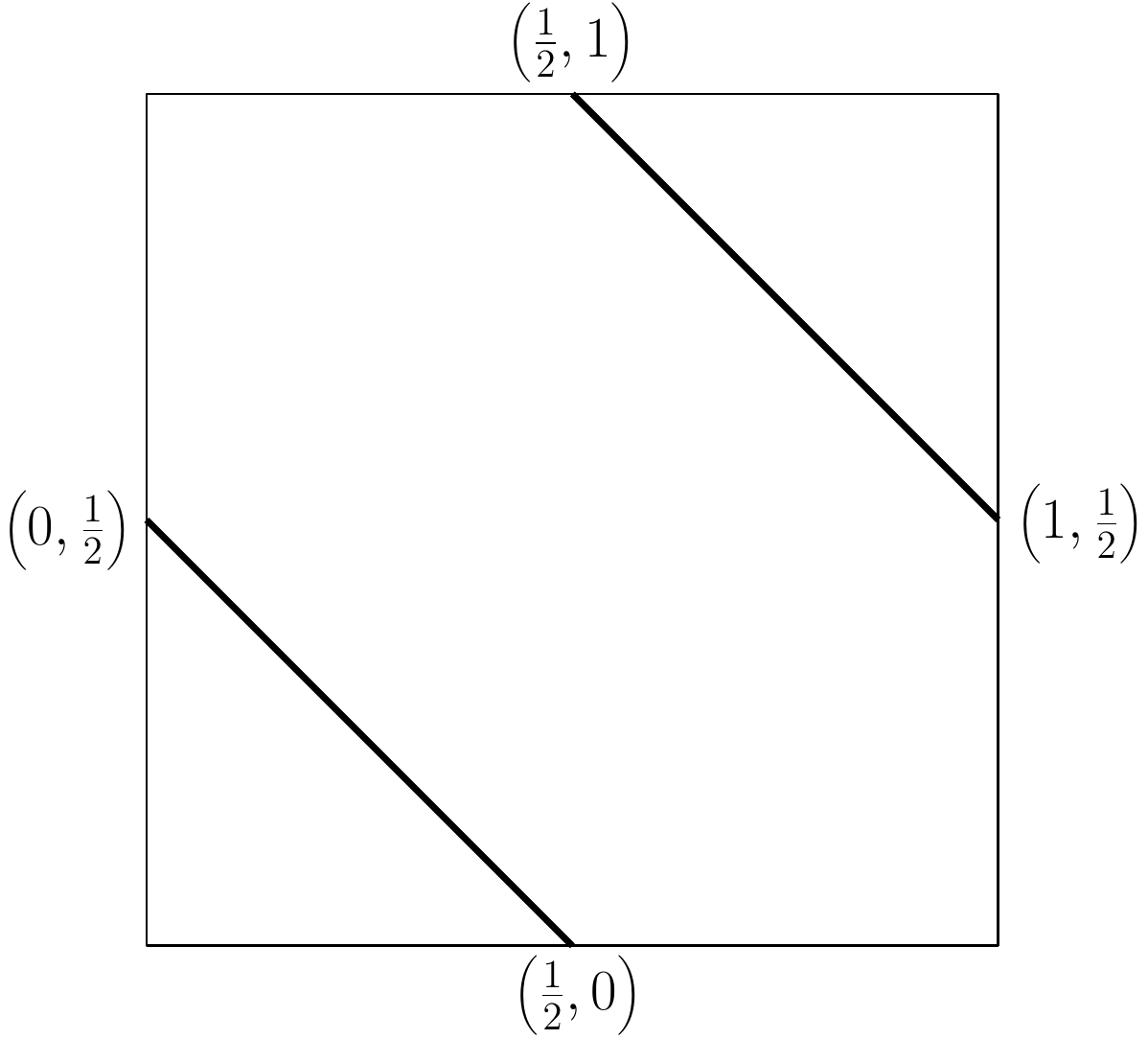}
		\caption{The support of the permuton $\xi$ introduced in Definition \ref{defn:counter}.}
		\label{fig:counter}
	\end{center}
\end{figure}

\begin{proposition}\label{prop:counter}
	Let $\xi$ be the permuton introduced in Definition \ref{defn:counter}. Then the set of optimizers $\mathcal{F}(21,\xi,\theta)$ of the optimization problem in \eqref{eq:opt} satisfies
	\begin{align*}
		\mathcal{F}(21,\xi,\theta)= \begin{cases}
			\{\xi\},  & \text{ if }\theta \le 1, \\
			\left\{ \dfrac{1+m_\theta}{2} \xi_{11}+\dfrac{1-m_\theta}{2}\xi_{22}, \dfrac{1-m_\theta}{2} \xi_{11}+\dfrac{1+m_\theta}{2}\xi_{22}\right\}, & \text{ if }\theta>1,
		\end{cases}
	\end{align*}
	where $m_\theta$ is the unique positive root of the equation $x=\tanh(\theta x)$ for $\theta>1$. 
	Moreover, the set $\mathcal{O}(\mathcal{F}(21,\xi,\theta))$ has cardinality 2 for all $\theta>1$.
\end{proposition}

\begin{proposition}\label{prop:xi}
	Let $\xi$ be the permuton introduced in Definition \ref{defn:counter}. Then $\xi$ is \cc\ and for all $\delta \in (t_{21}(\xi),\alpha_{21}(\xi))=(\frac12,1)$, and the set of optimizers $\mathcal{G}(21,\xi,\delta)$ of the optimization problem in \eqref{eq:f} satisfies
	\[\mathcal{G}(21,\xi,\delta)=\left\{\frac{1+\sqrt{2\delta-1}}{2}\xi_{11}+\frac{1-\sqrt{2\delta-1}}{2}\xi_{22}, \frac{1-\sqrt{2\delta-1}}{2}\xi_{11}+\frac{1+\sqrt{2\delta-1}}{2}\xi_{22}\right\}.\]
	Moreover, the set $\mathcal{O}(\mathcal{G}(21,\xi,\delta))$ has cardinality $2$ for all $\delta \in (t_{21}(\xi),\alpha_{21}(\xi))$.
\end{proposition}

\begin{remark}
	Proposition \ref{prop:counter} shows that there is a non trivial phase transition in the above model. On one hand this is not surprising, as the above model is essentially equivalent to the Curie--Weiss--Ising model (if we just track the box labels of the $n$ points), which has a non-trivial phase transition. (More explanations are given in the proof of Proposition \ref{prop:counter}.) In a similar manner, one can consider the following measure on the unit square: Partition the unit square into $q^2$ blocks of equal size, and  place an antidiagonal line on each of the $q$ diagonal blocks. This produces a Gibbs random permutation which is essentially equivalent to the Curie--Weiss--Potts model \cite{costeniuc2005complete}. The existence of a phase transition in these models is in contrast to what happens for the Mallows model, where there is no phase transition in the parameter $\theta$, and $\mathcal{F}(21,\lambda,\theta)$ has cardinality $1$ always (see \cite{Shannon,starr2018phase}). 
\end{remark}

\subsubsection{Interchanging conditioning events and limits}

As another application, we consider the following natural question. Consider the following two situations:
\begin{enumerate}[$(i)$]
	\item Start with a sequence of permutations $(\pi_{n,\sigma,\mu,\theta})_{n\ge 1}$, i.e., a sequence of Gibbs random permutations biased by its proportion of patterns $\sigma$, and condition the sequence to have an atypical proportion of inversions (compared to $\mathbb E[t_{\sigma}(\pi_{n,\sigma,\mu,\theta})]$);
	\item Assume now that $\theta\in(-\theta_c,\theta_c)$ and consider the limit of $\pi_{n,\sigma,\mu,\theta}$, that is $\nu_{\theta}:=\nu_{\sigma,\mu,\theta}$ (thanks to Theorem \ref{thm:gibbs_random2} part $(i)$). Now start with a sequence of  $\nu_{\theta}$-random permutations, that is, $(\pi_{n,\nu_{\theta}})_{n\ge 1}$, and condition the sequence to have the same atypical proportion of inversions as in situation $(i)$.
\end{enumerate}
Do the two conditioned sequences from situation $(i)$ and $(ii)$ have the same limit?
Our next result shows that the answer is negative already in the case when $\mu=\lambda$ is the Lebesgue measure on the unit square $[0,1]^2$ and $\sigma=21$ is an inversion.

\begin{theorem}\label{thm:gibbs4} 
	The following conclusions hold:
	\begin{enumerate}[(i)]
		\item Conditioned on the event $\{t_{\sigma}(\pi_{n,\sigma,\lambda,\theta}) \le \frac1{k!}\}$, the sequence $\pi_{n,\sigma,\lambda,\theta}$ converges in probability to $\lambda$.
		
		\item Let $\nu_{\theta}:=\nu_{21,\lambda,\theta}$. For all $\theta\in (0,\theta_c)$, the sequence $\pi_{n,\nu_\theta}$ conditioned on $\{t_{21}(\pi_{n,\nu_\theta})\le \frac1{2}\}$ converges in probability to some measure in $\widetilde{\cM}$, which is not the Lebesgue measure.
	\end{enumerate}
\end{theorem}

Note that our result gives a more precise answer to the previous question. Part $(i)$ of the above theorem shows that
$\pi_{n,\sigma,\lambda,\theta}$, after suitable conditioning and taking a limit, converges to Lebesgue measure. In contrast, part $(ii)$ shows for $\sigma=21$ if the operations of conditioning and limit are interchanged, then the limit is not Lebesgue measure.

\subsubsection{Existence of a phase transition for a generalized version of the Mallows model}

As already mentioned, it was shown in \cite{Shannon,starr2018phase} that if $\sigma=21$ and $\mu=\lambda$, i.e.\ in the Mallows model, then $\mathcal{O}(\mathcal{F}(21,\lambda,\theta))=\{\nu_\theta\}$ is a singleton for all $\theta\in \R$.
A simple combination with our results (see Theorem \ref{ppn:counter} part $(ii)$ below), implies that also $\mathcal{O}(\mathcal{G}(21,\lambda,\delta))$ is a singleton for all $\delta\in [t_{21}(\lambda),\alpha_{21}(\lambda))$.
The striking feature of this model is the absence of phase transitions. Our next theorem shows that this phenomenon does not take place when one changes the base measure $\mu$ from the Lebesgue measure $\lambda$ to  some other fully supported probability measure of the unit square.

\begin{definition}\label{def:counter2}
	We define a permuton $\mu_1$ in $\widetilde{\mathcal{M}}$ as follows. Suppose $X\sim U[0,1]$. If $X<\frac{1}{2}$, set $Y\sim U[0,1/2]$, and if $X>\frac{1}{2}$, set $Y\sim U[1/2,1]$. Then $(X,Y)$ is a random vector on $[0,1]^2$, and has uniform marginals. Let $\mu_1\in \widetilde{\mathcal{M}}$ denote the law of $(X,Y)$. The support of $\mu_1$ is shown  in the middle of Figure \ref{fig:phase_trans}.
	Let now $\mu_0=\lambda$ be Lebesgue measure on $[0,1]^2$, and for any $\ell\in [0,1]$, set
	\[\mu_\ell:=(1-\ell)\mu_0+\ell\mu_1.\]
	The support of $\mu_\ell$ is shown in the left-hand side  of Figure \ref{fig:phase_trans}.
\end{definition}

\begin{theorem}\label{ppn:counter}
	For all $\ell\in [0,1]$, let $\mu_\ell$ be the permuton introduced in Definition \ref{def:counter2}. Then the following conclusions hold:
	
	\begin{enumerate}[(i)]
		\item For all $\ell\in [0,1]$, the permuton $\mu_\ell$ is \cnc, $t_{21}(\mu_\ell)=\frac{2-\ell}{4}$, and $\alpha_{21}(\mu_\ell)=1$.
		
		\item If $\ell=0$, the set $\mathcal{O}(\mathcal{G}(21,\mu_0,\delta))$ has cardinality $1$ for all $\delta\in [1/2,1)$.
		
		\item For all $\ell\in [0,1]$, there exists $\delta_c(\ell)>\frac{2-\ell}{4}$ such that  the set $\mathcal{O}(\mathcal{G}(21, \mu_\ell,\delta))$ has cardinality $1$ if $\delta\in [\frac{2-\ell}{4}, \delta_c(\ell))$.
		
		\item There exists $\ell_c<1$ such that the following is true. For all $\ell\in(\ell_c,1]$, there exists $\delta'_c(\ell)<1$ such that
		the  set $\mathcal{O}(\mathcal{G}(21, \mu_\ell,\delta))$ has cardinality at least $2$ if $\delta\in(\delta'_c(\ell),1)$.
		
	\end{enumerate}
\end{theorem}

A schematic picture for the phase diagram explained in Theorem \ref{ppn:counter} is given on the right-hand side of Figure \ref{fig:phase_trans}. 

\begin{figure}[ht]
	\begin{center}
		\includegraphics[width=6.0in]{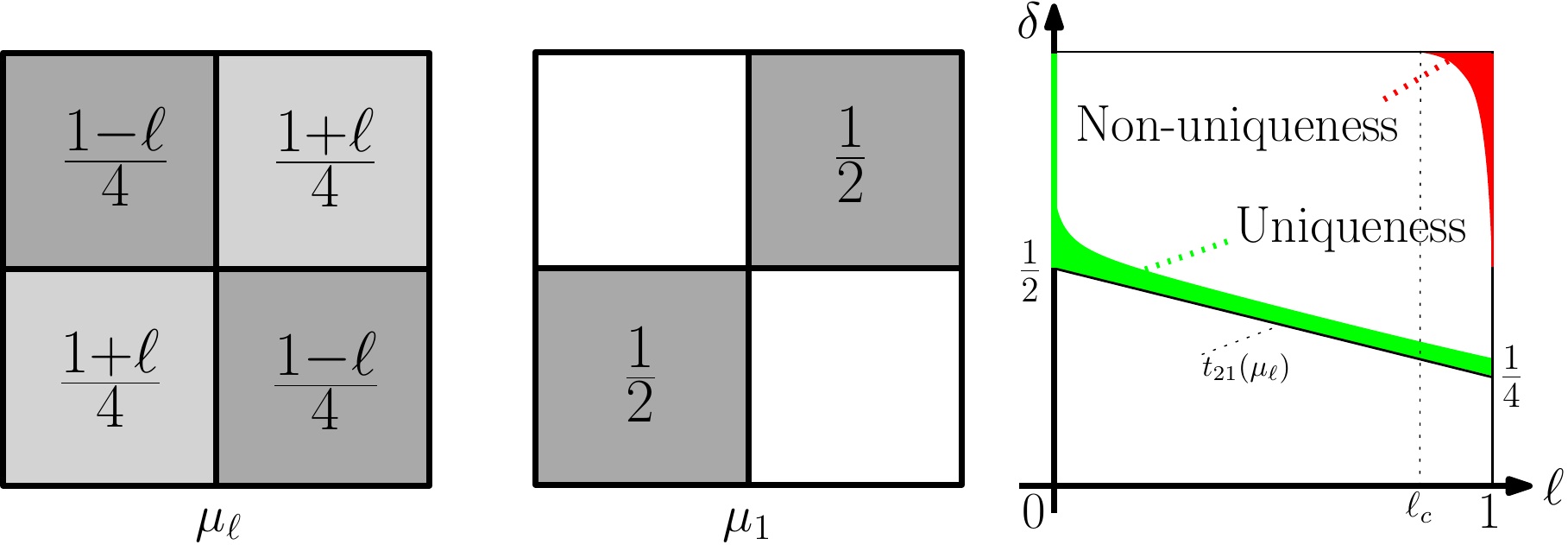}
		\caption{\textbf{Left:} The support of $\mu_\ell$ defined in Definition \ref{def:counter2}. Note that $\mu_0$ is equal to the Lebesgue measure on the unit square.  \textbf{Middle:} The support of $\mu_1$ defined in Definition \ref{def:counter2}.  \textbf{Right:} Schematic picture for the phase diagram explained in Theorem \ref{ppn:counter}.}
		\label{fig:phase_trans}
	\end{center}
\end{figure}

\subsubsection{\cc\ and \cnc\ measures with respect to inversions}

We now move to the problem of giving sufficient conditions for a measure $\mu$ to be \cnc\ with respect to inversions. We need to introduce the following definition.

\begin{definition}
	The (closed) \emph{support} of a probability measure $\mu$ on $[0,1]^2$, denoted by $\supp(\mu)$, is the intersection of all closed sets of $\mu$-measure $1$ (and hence it is closed). Equivalently, $\supp(\mu)$ is the set of all points $x\in[0,1]^2$ such that every open neighborhood of $x$ has positive measure.
\end{definition}

\begin{proposition}\label{prop:CNC21}
	Let $\mu\in\cM$ be a \cc\ measure with respect to $\sigma=21$. Then the following conclusions hold:
	\begin{enumerate}[(i)]
		\item The support $\supp(\mu)$ has empty interior.
	
	\bigskip
	
	\noindent We now further assume that $\mu$ is a permuton. Then we also have the following:
	
	\bigskip
	
		\item There exists a unique $b\in[0,1]$ depending on $\mu$ such that
		\begin{align*}
			\{(x,y)\in \supp(\mu) \mid \mu([0,x]\times[0,y])=0\} \subseteq	\{(x,y) \mid x+y=b\}.
		\end{align*}
		
		\item Consider $b$ from part $(ii)$. Let $\triangle=\{(x,y)\in [0,1]^2 \mid x\ge 0,y\ge 0,x+y<b\}$. Then $\mu(\triangle)=0$.
		
	\end{enumerate}
\end{proposition}

\begin{remark}
	Part (b) of Proposition \ref{prop:CNC21} says that if $(x,y)\in\supp(\mu)$ and $\mu([0,x]\times[0,y])=0$, then there is a line of slope $-1$ passing through $(x,y)$ so that every point in the support of $\mu$ is on the right of this line (see also Figure \ref{fig:Support_permutons_CC}). 
	The obvious symmetric property holds for the other three corners $(0,1), (1,0), (1,1)$ of the unit square, with a possibly different value of the parameter $b$.
\end{remark}

We refer the reader to Section \ref{sect:examples} for more examples of \cc\ permutons, for $\sigma=21$.

\begin{figure}[ht]
	\begin{center}
		\includegraphics[scale=0.6]{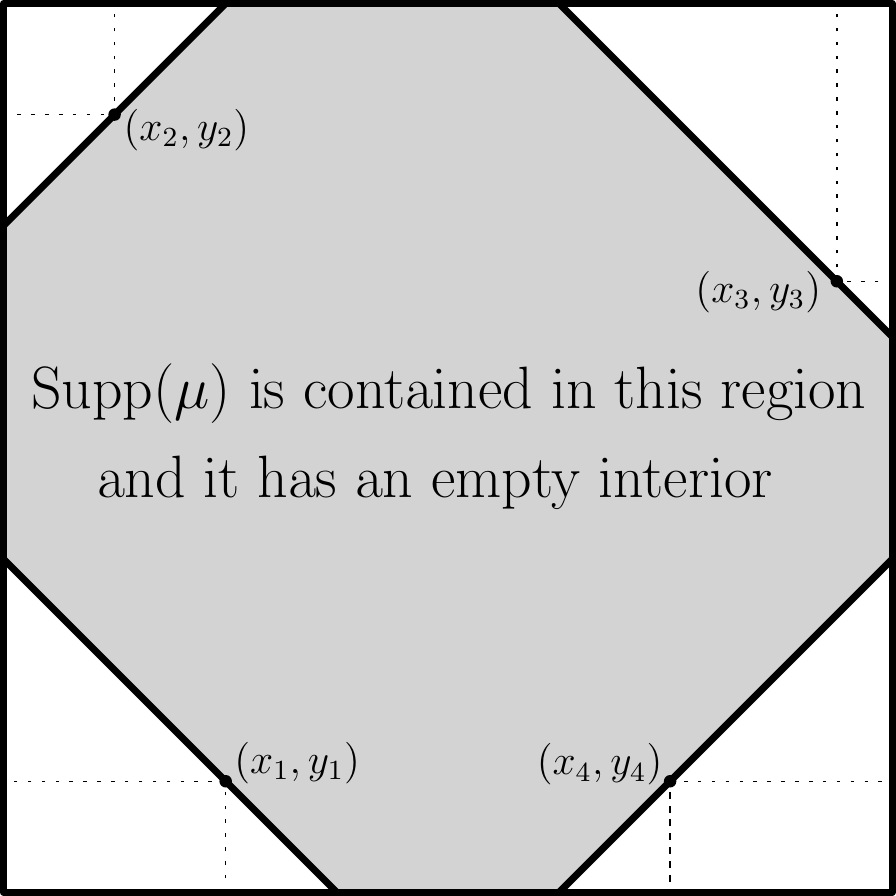}
		\caption{Schematic picture of a \cc\ permuton $\mu$ with respect to $\sigma=21$. 
			Proposition \ref{prop:CNC21} part $(ii)$ guarantees that the support of the permuton $\mu$ is contained in the gray region. We highlight that the diagonal boundaries of the gray region might collapse to the vertices  of the unit square $[0,1]^2$ (this situation correspond to the case when $b=0$ in the statement of Proposition \ref{prop:CNC21} part $(ii)$). For instance, the latter situation is realized when $\mu$ is supported either on the main diagonal or on the anti-diagonal.}
		\label{fig:Support_permutons_CC}
	\end{center}
\end{figure}

\subsection{Open problems}\label{openprob}

In this final section of the introduction we collect a list of open questions and problems that we think might be interesting to be addressed in future research projects.

\begin{itemize}
	\item For \cnc\ measures $\mu$, there is always a uniqueness phase (as shown in Theorem \ref{thm:constrain}), but there may or may not be a non-uniqueness phase, as demonstrated by parts $(ii)$, $(iii)$, $(iv)$ of Theorem \ref{ppn:counter}. We conjecture that the critical value $\ell_c$ of part $(iii)$ of Theorem \ref{ppn:counter} equals $0$, i.e., for any $\ell>0$ the measure $\mu_\ell$ has a uniqueness phase and a non-uniqueness phase, and so the nonexistence of the non-uniqueness phase for $\ell=0$ is somewhat special.
	
	\item It would be interesting to develop some more general tools to investigate problems such as the one in Theorem \ref{ppn:counter}. For instance, it would be interesting to find sufficient conditions to guarantee that the phase transition (if it exists) is unique. We point out that the main difficulty is that our optimization problems do not have linear constraints, but quadratic ones (for inversions) or of higher order (for general patterns).
	
	\item A challenging question is whether we can characterize/provide sufficient conditions for a permuton/measure $\mu$ to be \cc. This seems to be a complicated task even for the simplest pattern $\sigma=21$.  To the best of our knowledge, \cc\ permutons/measures for general patterns are hard to construct directly.
	
	\item As we have seen in the present paper, permutons are very useful to study various models of random permutations when one is interested in global properties of permutations, such as the proportion of patterns. On the other hand, permutons do not capture any of the local properties of the permutation (as shown in \cite[Theorem 4.1]{borga2020feasible} and \cite{bevan2022independence}). The first author of the present paper developed a theory of local convergence for permutations in \cite{borga2020local} in order to fill this gap. We believe that it would be interesting to study an LDP for permutations with respect to this new local topology in order to study some local properties such as the proportion of consecutive patterns (like ascents and descents). It might also be interesting to investigate some type of semi-local topologies such as the ones developed in \cite{bevan2022independence}.
	
	\item In Theorem \ref{thm:gibbs4} part $(ii)$, we proved that starting from an optimizer $\nu_{\theta}:=\nu_{21,\lambda,\theta}$ of the optimization problem in the right-hand side of \eqref{eq:opt} and considering the corresponding sequence $\pi_{n,\nu_\theta}$ conditioned on $\{t_{\sigma}(\pi_{n,\nu_\theta})\le \frac1{2}\}$, then this sequence does not converge to Lebesgue measure. We do expect that this results holds for any pattern $\sigma\in S_k$ (with the condition replaced by $\{t_{\sigma}(\pi_{n,\nu_{\sigma,\lambda,\theta}})\le \frac1{k!}\}$). The main obstacle to proving this conjecture for general patterns is that using the same strategy as in the proof of Theorem \ref{thm:gibbs4} part $(ii)$, one obtains some partial differential equations that for general patterns are more involved than the ones obtained for inversions.
\end{itemize}

\bigskip

\noindent\textbf{Organization of the paper.} The rest of the paper is organized as follows. In Section \ref{sec2} we prove our main results related to the general LDP presented in Section \ref{sec1.2} and then in Section \ref{sec3} we give the proofs of the applications to Gibbs random permutations introduced in Section \ref{sec1.3}.
Results related to constrained $\mu$-random permutations are proven in Section \ref{sec4}. In Section \ref{sec5}, we prove our results for inversions and provide several examples of \cc \ permutons with respect to inversions.
Proof of our main results involves applications of the general theory of large deviations, which is discussed in Appendix \ref{appA}. Finally Appendix \ref{appB} includes the proof of a technical lemma used in the paper.

\bigskip

\noindent\textbf{Acknowledgements.} SD was partially supported by NSF grant DMS-1928930 and Fernholz Foundation’s ``Summer Minerva Fellows'' program. SM was partially supported by NSF grants DMS-1712037 and DMS-2113414. The authors would like to thank Amir Dembo for helpful discussions. We thank two anonymous referees for some helpful comments. We thank Anirban Chatterjee for pointing out a mistake in the previous version of the paper.

\section{Proof of general large deviation principle} \label{sec2}

In this section we prove our main results about the general LDP from Section \ref{sec1.2}, i.e.\ Theorem \ref{thm:main}. We first derive a necessary lemma that establishes continuity of $t_{\sigma}(\cdot)$ under weak topology of measures.

\begin{lemma}\label{lem:rate}
	Suppose $(\nu_r)_{r\ge 1}$ is a sequence of measures in $\mathcal{P}$, such that $\nu_r$ converges weakly to $\nu_\infty\in \mathcal{M}$. Then for any pattern $\sigma\in S_k$, we have $t_\sigma(\nu_r)\to t_\sigma(\nu_\infty)$. 
\end{lemma}
\begin{proof}
	
	Since $\nu_r\stackrel{w}{\to}\nu_\infty$, we have $\nu_r^{\otimes k}\stackrel{w}{\to}\nu_\infty^{\otimes k}$. 
	Using Definition \ref{def:ts} we get
	$$t_\sigma(\nu_r)=\nu_r^{\otimes k}(h_\sigma)\to \nu_\infty^{\otimes k}(h_\sigma)=t_\sigma(\nu_\infty).$$
	In the above display, the convergence uses the fact that the function $h_\sigma$ is discontinuous on the set $$\Gamma:=\bigcup_{1\le a<b\le k}^k\Big[ \{x_a=x_b\}\cup \{y_a=y_b\}\Big]$$ and $\nu_\infty^{\otimes k}(\Gamma)=0$, as $\nu_\infty\in \mathcal{M}$ has continuous marginals.
\end{proof}

We first prove Theorem \ref{thm:main} and then show how Corollary \ref{cor:ldp_unif} can be deduced from it.

\begin{proof}[Proof of Theorem \ref{thm:main}]
	Let $\mu\in\mathcal M$. With $({\bf X,Y}):=(X_i,Y_i)_{i=1}^n \stackrel{i.i.d.}{\sim} \mu$,  by Sanov's theorem, the random empirical measure 
	$$\Delta_n=\Delta_n({\bf X},{\bf Y}):=\frac{1}{n}\sum_{i=1}^n\delta_{X_i,Y_i}$$
	satisfies an LDP on $\mathcal{P}$ with the good rate function $ D(\cdot|\mu)$. 
	Equip $[0,1]^\N$ with the product topology, and recall the map $\PT:\mathcal{P}\mapsto [0,1]^{\N}$ defined by
	\begin{equation}\label{eq:gen_map}
		\PT(\gamma)=(t_{\sigma_i}(\gamma))_{i\ge 1},
	\end{equation} 
	where $(\sigma_i)_{i\ge 1}$ is an enumeration of patterns of all sizes. By Lemma \ref{lem:rate}, the function $\PT(\cdot)$ is continuous on the set $\mathcal{M}$. Also, $D(\cdot|\mu)=\infty$ on $\mathcal{P}\setminus\mathcal{M}$, as any distributions with non-continuous marginals must be singular with respect to $\mu$. Thus by the contraction principle (part $(ii)$ of Lemma \ref{lem:ldp} with $\mathcal X=\mathcal{P}$, $X_n=\Delta_n$, $T=\PT$ and $J(\cdot)=D(\cdot|\mu)$), the sequence $\PT(\Delta_n)$ satisfies an LDP in $[0,1]^\N$ with respect to product topology, with speed $n$ and good rate function $G_\mu(\cdot)$ defined by
	\begin{eqnarray*}
		G_\mu({\bf t}):= \begin{cases}
			\inf_{\nu\in \mathcal{P}: \PT(\nu)={\bf t}} D(\nu|\mu)&\text{ if }{\bf t}\in \PT(\mathcal{P}),\\
			\infty &\text{ if }{\bf t}\notin \PT(\mathcal{P})
		\end{cases}.
	\end{eqnarray*} 
	Since $\mu\in{\cM}$, $D(\nu|\mu)<\infty$ only if $\nu \in {\mathcal{M}}$. 
	Also $\PT(\mathcal{M})=\PT(\widetilde{\cM})$, and so the good rate function above simplifies to:
	\begin{eqnarray}\label{eq:grf}
		G_\mu({\bf t})= \begin{cases}
			\inf_{\nu\in \cM: \PT(\nu)={\bf t}} D(\nu|\mu)&\text{ if }{\bf t}\in \PT(\widetilde{\mathcal{M}}),\\
			\infty &\text{ if }{\bf t}\notin \PT(\widetilde{\mathcal{M}}).
		\end{cases}
	\end{eqnarray} 
	Now, the construction in Definition \ref{def:mu_random} ensures that $\PT(\Em(\pi_{n,\mu}))=\PT(\Delta_n)$, and so $\PT(\Em(\pi_{n,\mu}))$ satisfies an LDP with speed $n$, and good rate function $G_\mu(\cdot)$. Finally by \cite{Hoppen}, the map $\PT$ restricted to $\widetilde{\mathcal{M}}$ is 1-1, and $ \PT^{-1}$ is continuous from $\PT(\widetilde{\cM})$ to $\widetilde{\cM}$.
	It then follows by another application of the contraction principle (part $(ii)$ of Lemma \ref{lem:ldp} with $\mathcal X=\PT(\widetilde{\mathcal{M}})$, $X_n=\PT(\Em(\pi_{n,\mu}))$, $T=\PT^{-1}$ and $J(\cdot)=G_\mu(\cdot)$) that $\Em(\pi_{n,\mu})$ satisfies an LDP with speed $n$ and good rate function $J_\mu(\cdot)$, defined by
	\begin{eqnarray*}
		J_\mu(\gamma):= \begin{cases}
			\inf_{{\bf t}\in \PT(\widetilde{\mathcal{M}}): \PT(\gamma)={\bf t}} G_\mu({\bf t})&\text{ if }\gamma\in \widetilde{\mathcal{M}},\\
			\infty&\text{ if }\gamma\notin \til{\mathcal{M}}
		\end{cases}&.
	\end{eqnarray*}
	Recalling the expression for $G_\mu({\bf t})$ in \eqref{eq:grf} and noting that 
	$$\inf_{{\bf t}\in \PT(\widetilde{\mathcal{M}}): \PT(\gamma)={\bf t}} \inf_{\nu\in\cM: \PT(\nu)={\bf t}} D(\nu|\mu)=\inf_{\nu \in\cM: \PT(\gamma)=\PT(\nu)}D(\nu|\mu)=\inf_{\nu\in\cM: \mathcal{O}(\nu)=\gamma} D(\nu|\mu),$$
	we have $I_\mu=J_\mu$, and so the desired conclusion of the theorem follows.
\end{proof}

\begin{proof}[Proof of Corollary \ref{cor:ldp_unif}]
	To recover Corollary \ref{cor:ldp_unif} from Theorem \ref{thm:main}, it suffices to note if $\mu=\lambda$ is the Lebesgue measure on $[0,1]^2$, then for any $\gamma\in \widetilde{\cM}$ by \cite[Proposition 4.2]{starr2018phase} we have
	$\inf_{\nu\in \cM: \mathcal{O}(\nu)=\gamma} D(\nu|\lambda)=D(\gamma|\lambda).$ Consequently, $I_\lambda(\gamma)=I(\gamma)$. 
\end{proof}

We conclude this section with the proof of Proposition \ref{ppn:counter_new}.

\begin{proof}[Proof of Proposition \ref{ppn:counter_new}]
	
	Take $p\in (0,1)$ with $p\neq \frac12$. Define $$ \nu:=4p\cdot \lambda([0,\tfrac12]^2)+4(1-p)\cdot \lambda([\tfrac12,1]^2), \quad \gamma:=p^{-1}\cdot \lambda([0,p]^2)+(1-p)^{-1} \cdot \lambda ([p,1]^2).$$ We claim that $\mathcal{O}(\nu)=\gamma$. To see this, let $Z\sim \operatorname{Ber}(p)$. If $Z=1$, set $X,Y \stackrel{i.i.d.}{\sim} \operatorname{Unif}[0,\frac12]$, else $X,Y \stackrel{i.i.d.}{\sim} \operatorname{Unif}[\frac12,1]$. Clearly $(X,Y)\sim \nu$ and they have a common marginal c.d.f., say $F$. We have $(F(X),F(Y)|Z=1) \stackrel{i.i.d.}{\sim}  \operatorname{Unif}[0,p]$, and $(F(X),F(Y)|Z=0) \stackrel{i.i.d.}{\sim} \operatorname{Unif}[p,1]$. Thus, $(F(X),F(Y))\sim \gamma$. However, $D(\nu|\mu_1)<\infty$ and $D(\gamma|\mu_1)=\infty$.  
\end{proof}

\section{Proofs of the applications to Gibbs random permutations}\label{sec3}

We now turn towards proving the main results about Gibbs random permutations, namely, Theorem \ref{thm:gibbs_random} and Theorem \ref{thm:gibbs_random2}.

\begin{proof}[Proof of Theorem \ref{thm:gibbs_random}]
	
	$(i)$ Note that by \eqref{eq:gibbs_log_part} and recalling \eqref{eq:extend_map},
	\begin{equation*}
		\frac{F_n(\sigma,\mu,\theta)}{n}=\frac{1}{n}\log\E\exp\left(n\theta
		t_\sigma(\Em(\pi_{n,\mu})) \right).
	\end{equation*}
	Recall that $\Em(\pi_{n,\mu})$ satisfies an LDP by Theorem \ref{thm:main} with the good rate function $I_{\mu}(\cdot)$ introduced in \eqref{eq:rate}, and the map $\nu\mapsto t_\sigma(\nu)$ is continuous on $\mathcal{M}$ by Lemma \ref{lem:rate}. In particular, the map $t_\sigma(\cdot)$ is continuous where $I_{\mu}(\cdot)$ is finite. It follows by an application of Varadhan's Lemma (part $(iii)$ of Lemma \ref{lem:ldp} with $\mathcal{X}=\mathcal{M}$, $X_n=\Em(\pi_{n,\mu})$, $T=\theta t_\sigma$ and $J=I_{\mu}$) that
	\begin{multline}
		\lim_{n\to\infty}\frac{F_n(\sigma,\mu,\theta)}{n}=\sup_{\gamma\in \cM}\{\theta t_\sigma(\gamma)-I_\mu(\gamma)\}=\sup_{\gamma\in \widetilde{\cM}} \{\theta t_\sigma(\gamma)-\inf_{\nu\in \cM:\mathcal{O}(\nu)=\gamma} D(\nu|\mu)\}\\
		=\sup_{\gamma\in \widetilde{\cM}}\, \sup_{\nu\in \cM:\mathcal{O}(\nu)=\gamma}\{\theta t_\sigma(\nu)-D(\nu|\mu)\}
		=\sup_{\nu\in \cM}\{\theta t_\sigma(\nu)-D(\nu|\mu)\},
	\end{multline}
	where in the second equality we used that $I_{\mu}(\gamma)=\infty$ if $\gamma\notin \widetilde{\cM}$.
	
	\medskip
	
	$(ii)$	Since $\sup_{\nu\in \mathcal{M}}\{\theta t_\sigma(\nu)-D(\nu|\mu)\}$ is finite, there exists $\alpha<\infty$ such that it suffices to solve the optimization for $\nu$ such that $$\theta t_\sigma(\nu)-D(\nu|\mu)\le \alpha\Rightarrow D(\nu|\mu)\le |\theta|+\alpha.$$
	Since $D(.|\mu)$ is a good rate function, the set $\{\nu\in \mathcal{M}:D(\nu|\mu)\le |\theta|+\alpha\}$ is compact. Also, the function $\nu\mapsto \theta t_\sigma(\nu)-D(\nu|\mu)$ is continuous on this set by Lemma \ref{lem:rate}, and so the supremum in part $(i)$ is attained.

	For finding the optimizers, it is enough to look among $\nu$'s which are absolutely continuous with respect to $\mu$. Let us assume $\nu$ is one of the maximizers, and $\frac{\d\nu}{\d \mu}=g$. 
	Take another measure which is absolutely continuous with respect to $\mu$ and denote its Radon–Nikodym derivative by $\tilde{g}$. Define $g_c:=(1-c)g+c\tilde{g}=g+c(\tilde{g}-g)$, and note that $\nu_c:=g_cd\mu\in \mathcal{M}$ and that $\partial/\partial c (g_c)=\tilde{g}-g$. Setting
	\begin{align*}
		f(c):=\theta t_\sigma(\nu_c)-D(\nu_c|\mu)=	\theta \int_{[0,1]^{2k}} h_\sigma(z_1,\ldots,z_k)\prod_{a=1}^k g_c(z_a)\d\mu(z_a)-\int_{[0,1]^2} g_c(z)\log g_c(z)\d\mu(z)
	\end{align*}
	and using the fact that $\nu_0=gd\mu$ is an optimizer, it follows that $f'(0)\le 0$, which gives
	\begin{align}\label{eq:elmid}
		\int_{[0,1]^2} (\tilde{g}(z_1)-g(z_1))\left[k\theta \int_{[0,1]^{2k-2}} h_\sigma(z_1,\ldots,z_k)\prod_{a=2}^k g(z_a)\d\mu(z_a)-\log g(z_1)\right]\d\mu(z_1) \le 0.
	\end{align}
	Note that \eqref{eq:elmid} implies
	\begin{align*}
		\int_{[0,1]^2} \til{g}(z_1)\log g(z_1)\d\mu(z_1) \ge D(\nu|\mu)-2k\theta \implies \int_{[0,1]^2} \til{g}(z_1)\log [e^{2k\theta}g(z_1)]\d\mu(z_1)\ge 0.
	\end{align*}
	The above inequality holds for any $\tilde{g}\ge 0$ such that $\int_{[0,1]^2}\tilde{g}d\mu=1$, and consequently, 
	by Theorem 1.6.11 in \cite{ash} we have $g(z_1)\ge e^{-2k\theta}$ $\mu$-almost surely. For any $\phi\in [0,1]^2\to [-1,1]$ with $\int_{[0,1]^2}\phi(z)\mu(z)=0$, observe that $\til{g}(z):=e^{-2k\theta}\phi(z)+g(z) \ge 0$ $\mu$-almost surely and $\int_{[0,1]^2} \til{g}(z)\d\mu(z)=1$. Thus for this choice of $\til{g}$ in \eqref{eq:elmid} we have
	\begin{align*}
		\int_{[0,1]^2} \phi(z_1)\left[k\theta \int_{[0,1]^{2k-2}} h_\sigma(z_1,\ldots,z_k)\prod_{a=2}^k g(z_a)\d\mu(z_a)-\log g(z_1)\right]\d\mu(z_1) \le 0.
	\end{align*}
	The above inequality holds for any $\phi:[0,1]^2\mapsto [-1,1]$ such that $\int_{[0,1]^2}\phi(z)d\mu(z)=0$. Changing $\phi$ to $-\phi$, we conclude
	\begin{align*}
		\int_{[0,1]^2} \phi(z_1)\left[k\theta \int_{[0,1]^{2k-2}} h_\sigma(z_1,\ldots,z_k)\prod_{a=2}^k g(z_a)\d\mu(z_a)-\log g(z_1)\right]\d\mu(z_1) = 0.
	\end{align*}
	Since this must happen for all $\phi:[0,1]^2\mapsto [-1,1]$ such that $\int_{[0,1]^2}\phi(z)d\mu(z)=0$, we get that for $\mu$-almost every $z_1$
	\begin{align*}
		k\theta \int_{[0,1]^{2k-2}} h_\sigma(z_1,\ldots,z_k)\prod_{a=2}^k g(z_a)\d\mu(z_a)-\log g(z_1)=\mbox{constant}.
	\end{align*}
	The desired Euler-Lagrange equation follows from this equation and the fact that $\int_{[0,1]^2} g d\mu=1$.
	
	\medskip
	
	$(iii)$
	A direct calculation gives
	\begin{align}\label{eq:gibbs_lebesgue_gen}
		\notag\P(\pi_{n,\sigma,\mu,\theta}=\tau)=&e^{n\theta t_\sigma(\tau)-F_n(\sigma,\mu,\theta)} \int_{[0,1]^{2n}} \mathds{1}\{\pi_{{\bf x},{\bf y}}=\tau\}  \prod_{i=1}^n d\mu(x_i,y_i)\\
		= &e^{n\theta t_\sigma(\tau)-F_n(\sigma,\mu,\theta)} \P(\pi_{n,\mu}=\tau).
	\end{align}
	From Theorem \ref{thm:main}, $\Em(\pi_{n,\mu})$ satisfies an LDP with speed $n$ and good rate function $I_\mu(\gamma)$. Moreover, the function $\gamma\mapsto \theta t_\sigma(\gamma)$ is continuous on the set $\widetilde{\cM}$ by Lemma \ref{lem:rate}, and $I_\mu(\gamma)=\infty$ outside $\widetilde{\mathcal{M}}$. Therefore, from the expression in \eqref{eq:gibbs_lebesgue_gen}, we can invoke Lemma \ref{lem:ldp} part $(iv)$ (with $\mathcal{X}=\mathcal{P}$, $X_n=\Em(\pi_{n,\mu})$, $T(\cdot)=\theta t_\sigma(\cdot)$, $J(\cdot)=I_\mu(\cdot)$ and $Y_n=\Em(\pi_{n,\sigma,\mu,\theta})$ and deduce that $\Em(\pi_{n,\sigma,\mu,\theta})$ satisfies an LDP with speed $n$, and good rate function
	\begin{eqnarray*} 
		I_{\sigma,\mu,\theta}(\gamma)=\begin{cases}
			I_\mu(\gamma)-\theta t_\sigma(\gamma)-\inf_{\gamma \in \widetilde{\cM}}\{I_\mu(\gamma)-\theta t_\sigma(\gamma)\}&\text{ if }\gamma\in \widetilde{\mathcal{M}},\\
			\infty&\text{ if }\gamma\notin \widetilde{\mathcal{M}}
		\end{cases}.
	\end{eqnarray*}
	The desired form of the rate function follows on recalling the definition of $I_\mu(\cdot)$, from which we obtain that
	$$I_\mu(\gamma)-\theta t_\sigma(\gamma)= \inf_{\nu\in {\cM}: \mathcal{O}(\nu)=\gamma}\{D(\nu|\mu)-\theta t_\sigma(\nu)\}$$
	and noting that
	$$
	\inf_{\gamma\in \widetilde{\cM}}\{I_\mu(\gamma)-\theta t_\sigma(\gamma)\}=\inf_{\nu \in \cM}\{D(\nu|\mu)-\theta t_\sigma(\nu)\}.$$
	
	\medskip
	
	$(iv)$  Note that the term $\inf_{\nu \in \cM} \{D(\nu|\mu)-\theta t_\sigma(\nu)\}$ in the the expression of $I_{\sigma,\mu,\theta}(\gamma)$ is independent of $\gamma$. Therefore, using part $(iii)$ above, along with part $(i)$ of Lemma \ref{lem:ldp}, we have that $\pi_{n,\sigma,\mu,\theta}$ converges to the set of minimizers of the function
	$$\gamma\mapsto \inf_{\nu\in \cM:\mathcal{O}(\nu)=\gamma}\{D(\nu|\mu)-\theta t_\sigma(\nu)\},$$ which is just $\mathcal{O}(\mathcal{F}(\sigma,\mu,\theta))$.
\end{proof}

We now move to the proof of Theorem \ref{thm:gibbs_random2}, i.e., the one regarding Gibbs random permutations in the so-called high-temperature phase. In the next proof, given a probability measure $\mu\in\mathcal P$, $\|.\|_1$ denotes the $L^1(\mu)$ norm of a function and $\|.\|_\infty$ the sup norm. We also denote by $TV$ the total variation distance.

\begin{proof}[Proof of Theorem \ref{thm:gibbs_random2}]
	
	In the first four parts, the dependence of $\nu_{\sigma,\mu,\theta}$ on $(\sigma,\mu)$  is not important, so for simplicity of notation, we will denote $\nu_{\sigma,\mu,\theta}$ simply by $\nu_\theta$. 
	
	$(i)$ Set $\mathcal{L}_1(\mu)=\left\{g:\int gd\mu=1,g\geq 0 \right\}$. For all $\theta\in\R$ define a map $T_{\theta}:\mathcal{L}_1(\mu)\mapsto \mathcal{L}_1(\mu)$ by setting
	\begin{align*}
		T_{\theta}(u(z_1)):=\frac{\exp\Big(k\theta \int_{[0,1]^{2k-2}} h_\sigma(z_1,\ldots,z_k)\prod_{a=2}^k u(z_a)\d \mu(z_a)\Big)}{\int_{[0,1]^2} \exp\Big(k\theta \int_{[0,1]^{2k-2}} h_\sigma(z_1,\ldots,z_k)\prod_{a=2}^k u(z_a)\d \mu(z_a)\Big)\d \mu(z_1)}.
	\end{align*}
	We claim that if $\theta_c:=\min\left\{C^{-1},1\right\}$ where $C:=\sup_{|x|\le 4k^2} \Big|\frac{e^x-1}{x}\Big|$, then for all $\theta\in (-\theta_c,\theta_c)$, the map $T_{\theta}$ is a contraction in $\mathcal{L}_1(\mu)$. Note that this is enough to show that for all $\theta\in (-\theta_c,\theta_c)$ there is a unique solution to the Euler-Lagrange equation of part $(ii)$ of Theorem \ref{thm:gibbs_random}. We prove our claim. Take $u,v \in \mathcal{L}_1(\mu)$ and note that
	\begin{align*}
		\int_{[0,1]^{2k-2}}\left|\prod_{i=2}^k u(z_i)-\prod_{i=2}^k v(z_i)\right|\prod_{a=2}^k\d\mu(z_a)\leq\sum_{i=2}^k\int_{[0,1]^{2}}|u(z_i)-v(z_i)|d\mu(z_i)\leq k \lVert u-v\rVert_{1}.
	\end{align*}
	where we used that $\prod_{i=2}^k u(z_i)-\prod_{i=2}^k v(z_i)=\sum_{i=2}^k u(z_2)\dots u(z_{i-1})(u(z_{i})-v(z_{i}))v(z_{i+1})\dots v(z_{k})$.
	As a consequence, since $||h_{\sigma}||_{\infty}\le 1$, for all $z_1\in [0,1]^2$ we have
	\begin{align*}
		-\theta k \lVert u-v\rVert_{1}\le \theta\int_{[0,1]^{2k-2}} h_\sigma(z_1,\ldots,z_k)\left[\prod_{i=2}^k u(z_i)-\prod_{i=2}^k v(z_i)\right]\prod_{a=2}^k\d\mu(z_a)
		\le \theta k \lVert u-v\rVert_{1}.
	\end{align*}
	This implies that 
	\begin{align*}
		e^{-\theta k^2 \lVert u-v\rVert_{1}} \le \frac{\exp\Big(k\theta \int_{[0,1]^{2k-2}} h_\sigma(z_1,\ldots,z_k)\prod_{a=2}^k u(z_a)\d \mu(z_a)\Big)}{\exp\Big(k\theta \int_{[0,1]^{2k-2}} h_\sigma(z_1,\ldots,z_k)\prod_{a=2}^k v(z_a)\d \mu(z_a)\Big)} \le e^{\theta k^2 \lVert u-v\rVert_{1}}.
	\end{align*}
	which in turn gives
	\begin{align*}
		e^{-\theta k^2 \lVert u-v\rVert_{1}} \le \frac{\int_{[0,1]^2} \exp\Big(k\theta \int_{[0,1]^{2k-2}} h_\sigma(z_1,\ldots,z_k)\prod_{a=2}^k u(z_a)\d \mu(z_i)\Big)\d\mu(z_1)}{\int_{[0,1]^2} \exp\Big(k\theta \int_{[0,1]^{2k-2}} h_\sigma(z_1,\ldots,z_k)\prod_{a=2}^k v(z_a)\d \mu(z_a)\Big)\d\mu(z_1)} \le e^{\theta k^2 \lVert u-v\rVert_{1}}.
	\end{align*}
	Taking ratios of the last two displayed equations, we get
	\begin{align*}
		e^{-2\theta k^2 \lVert u-v\rVert_{1}} \le \frac{T_\theta(u(z_1))}{T_\theta(v(z_1))} \le e^{2\theta k^2 \lVert u-v\rVert_{1}},
	\end{align*}
	and subtracting 1 from each term in the above display and then multiplying by $T_\theta(v(z_1))$ we get
	\begin{align*}
		T_\theta(v(z_1))\left(e^{-2\theta k^2 \lVert u-v\rVert_{1}}-1\right) \le T_\theta(u(z_1))-T_\theta(v(z_1)) \le T_\theta(v(z_1))\left(e^{2\theta k^2 \lVert u-v\rVert_{1}}-1\right),
	\end{align*}
	Recalling that $C=\sup_{|x|\le 4k^2} \Big|\frac{e^x-1}{x}\Big|$, and noting that $|\theta|\leq |\theta_c|\leq 1$ and that $\lVert u-v\rVert_{1}\leq 2$ (because $u,v \in \mathcal{L}_1(\mu)$), we get
	\begin{align*}
		\Big|T_\theta(u(z_1))-T_\theta(v(z_1))\Big| \le C |\theta|\cdot ||u-v||_{1}\cdot  T_\theta(v(z_1)).
	\end{align*}
	Since $T_\theta(v)\in \mathcal{L}_1(\mu)$, we conclude that $||T_\theta (u)-T_\theta (v)||_{1} \le C\lvert\theta\rvert \lVert u-v\rVert_{1}$. Thus for  $\theta\in (-\theta_c,\theta_c)$  the map $T_\theta$ is a contraction. 
	
	The fact that $\pi_{n,\sigma,\mu,\theta}$ converges in probability to $\mathcal{O}(\nu_{\theta})$ then simply follows from Theorem \ref{thm:gibbs_random} part $(iv)$.

	\medskip
	
	$(ii)$	To show continuity, we take any sequence $(\theta_r)_{r\ge 1}\in (-\theta_c,\theta_c)$, such that $\theta_r \to \theta\in (-\theta_c,\theta_c)$. Since $\nu_\theta$ is an optimizer of \eqref{eq:opt}, for any $\theta$ we have
	$$\theta t_\sigma(\nu_\theta)-D(\nu_\theta|\mu)\ge \theta t_\sigma(\mu)\Rightarrow D(\nu_\theta|\mu)\le \theta (t_\sigma(\nu_\theta)-t_\sigma(\mu)).$$
	This shows that $\limsup_{r\to\infty}D(\nu_{\theta_r}|\mu)<\infty$. Also note that $\nu_{\theta_r}\in \mathcal{M}\subset\mathcal{P}$ is a tight sequence, and so there exists  $\nu_\infty\in \mathcal{P}$ which is a weak limit point of this sequence. By 
	lower semi-continuity of $D(\cdot|\mu)$ we have $D(\nu_\infty|\mu)<\infty$, and so $\nu_\infty\in \mathcal{M}$. 
	Continuity of $F(\sigma,\mu,\cdot)$ (because it is the limit of convex functions), lower semi-continuity of $D(\cdot|\mu)$, and continuity of $t_\sigma(\cdot)$ at $\nu_\infty\in \mathcal{M}$ (by Lemma \ref{lem:rate}), then gives
	\begin{align*}
		F(\sigma,\mu,\theta) & = \lim_{r\to\infty} F(\sigma,\mu,\theta_r) = \lim_{r\to\infty} \theta_r t_\sigma(\nu_{\theta_r})-D(\nu_{\theta_r}|\mu) \le \theta t_\sigma(\nu_\infty)-D(\nu_{\theta_\infty}|\mu),
	\end{align*}
	where in the second equality we used that $\nu_{\theta_r}$ is the unique optimizer of $F(\sigma,\mu,\theta_r)$ by part $(i)$. 
	Since $\nu_\theta$ is the unique maximizer of $F(\sigma,\mu,\theta)$, we must have $\nu_\infty=\nu_{\theta}$. This gives \[t_\sigma(\nu_\theta)=t_\sigma(\nu_\infty)=\lim_{r\rightarrow\infty}t_\sigma(\nu_{\theta_r})\] and so $\theta\mapsto t_\sigma(\nu_\theta)$ is continuous. 
	
	Proceeding to show monotonicity, using parts $(i)$ of Theorem \ref{thm:gibbs_random} and Theorem \ref{thm:gibbs_random2}, we have  $$\frac{F_n(\sigma,\mu,\theta)}{n}\to F(\sigma,\mu,\theta)=\theta t_{\sigma}(\nu_\theta)-D(\nu_\theta|\mu).$$
	Also, if $({\bf X},{\bf Y})\sim Q_{n,\sigma,\mu,\theta}$ , then it follows from part $(i)$ above, and parts $(i)$ and $(iv)$ of Lemma \ref{lem:ldp} that $\Em(\pi_{{\bf X},{\bf Y}})\stackrel{w}{\to}\nu_\theta$ in probability, and so since $t_\sigma(\cdot)$  is bounded and continuous at $\nu_\infty\in \mathcal{M}$ (by Lemma \ref{lem:rate}) we get
	$$\frac{1}{n} F_n'(\sigma,\mu,\theta)\stackrel{\eqref{eq:gibbs_log_part}}{=}\E [t_\sigma(\pi_{n,\sigma,\mu,\theta})]=\E [t_\sigma(\Em(\pi_{{\bf X},{\bf Y}}))]\to t_\sigma(\nu_\theta).$$
	where in the second equality we used the definitions in \eqref{eq:Gibss_def_perm} and \eqref{eq:extend_map}. An application of dominated convergence theorem yields
	\begin{align*}
		F(\sigma,\mu,\theta+h)-F(\sigma,\mu,\theta)  =\lim_{n\to\infty} \frac1n[F_n(\sigma,\mu,\theta+h)-F_n(\sigma,\mu,\theta)] & =\lim_{n\to\infty}\int_{\theta}^{\theta+h} \frac1nF_n'(\sigma,\mu,\alpha)\d \alpha \\ & = \int_{\theta}^{\theta+h}t_\sigma(\nu_{\alpha})\d\alpha.
	\end{align*}
	This implies $F(\sigma,\mu,\theta)$ is differentiable with $F'(\sigma,\mu,\theta)=t_\sigma(\nu_\theta)$. The fact that $t_\sigma(\nu_\theta)$ is non-decreasing follows from the observation that $F_n(\sigma,\mu,\cdot)$ is convex.
	
	\medskip
	
	$(iii)$  Assume $\theta_1,\theta_2 \in (-\theta_c,\theta_c)$ are such that $t_{\sigma}(\nu_{\theta_1})=t_{\sigma}(\nu_{\theta_2})$. Using the fact that $\nu_{\theta_1}$ is the optimizer of $\sup_{\nu\in\cM} \{\theta_1 t_\sigma(\nu)-D(\nu|\mu)\}$, we have 
	\begin{align*}
		\theta_1t_{\sigma}(\nu_{\theta_1})-D(\nu_{\theta_1}|\mu) \ge \theta_1t_{\sigma}(\nu_{\theta_2})-D(\nu_{\theta_2}|\mu)
	\end{align*}
	which implies $D(\nu_{\theta_2}|\mu) \ge D(\nu_{\theta_1}|\mu)$. By symmetry $D(\nu_{\theta_2}|\mu)=D(\nu_{\theta_1}|\mu)$ and so $ \theta_1t_\sigma(\nu_{\theta_1})-D(\nu_{\theta_1}|\mu)= \theta_1t_\sigma(\nu_{\theta_2})-D(\nu_{\theta_2}|\mu)$. Since the optimizing measure is unique (by part $(i)$), this forces $\nu_{\theta_1}=\nu_{\theta_2}$.
	
	\medskip
	
	$(iv)$ Let $(\theta_m)_{m\ge1}$ be a sequence converging to $\theta_\infty\in (-\theta_c,\theta_c)$. We will show that $\nu_{\theta_m}\stackrel{TV}{\rightarrow}\nu_{\theta_\infty}$. To this effect, set $f_m(\cdot):=\frac{d\nu_{\theta_m}}{d\mu}$ for $1\le m\le \infty$, and use part $(ii)$ of Theorem \ref{thm:gibbs_random} to recall that
	\begin{align}\label{eq:el}
		f_{m}(z_1)=\frac{\exp\Big(k\theta_m \int_{[0,1]^{2k-2}}h_\sigma(z_1,\ldots,z_k) \prod_{a=2}^k f_m(z_a) d\mu(z_a)\Big)}{\int_{[0,1]^2} \exp\Big(k\theta_m \int_{[0,1]^{2k-2}}h_\sigma(z_1,\ldots,z_k) \prod_{a=2}^k f_m(z_a) d\mu(z_a)\Big)}.
	\end{align}
	A direct computation gives
	\begin{align*}
		& \left|\theta_m \int_{[0,1]^{2k-2}} h_\sigma(z_1,\ldots,z_k)\prod_{a=2}^k f_m(z_a)d\mu(z_a)-\theta_\infty \int_{[0,1]^{2k-2}} h_\sigma(z_1,\ldots,z_k)\prod_{a=2}^k f_\infty(z_a)d\mu(z_a)\right| \\ & 
		\leq\left|(\theta_m-\theta_\infty) \int_{[0,1]^{2k-2}} h_\sigma(z_1,\ldots,z_k)\prod_{a=2}^k f_m(z_a)d\mu(z_a)\right|\\
		&+\left|\theta_\infty \int_{[0,1]^{2k-2}} h_\sigma(z_1,\ldots,z_k)\prod_{a=2}^k (f_m(z_a)-f_\infty(z_a))d\mu(z_a)\right| \\ & 
		\le |\theta_m-\theta_\infty|+|\theta_\infty k |. \|f_m-f_\infty\|_{1}.
	\end{align*}	
	Using \eqref{eq:el} and taking ratios gives (as in the proof of part $(i)$)
	\begin{align*}
		e^{-2k|\theta_m-\theta_\infty|-2k^2|\theta_\infty|\cdot \|f_m-f_\infty\|_{1}} \le \frac{f_m(z)}{f_\infty(z)} \le e^{2k|\theta_m-\theta_\infty|+2k^2|\theta_\infty|\cdot\|f_m-f_\infty\|_{1}}.
	\end{align*}
	Subtracting one from each term in the above display and then multiplying by $f_\infty(z)$ we get
	\begin{align*}
		f_\infty(z)(e^{-2k|\theta_m-\theta_\infty|-2k^2|\theta_\infty| \cdot \|f_m-f_\infty\|_{1}}-1) \le f_m(z)-f_\infty(z) \le (e^{2k|\theta_m-\theta_\infty|+2k^2|\theta_\infty|\cdot \|f_m-f_\infty\|_{1}}-1)f_\infty(z).
	\end{align*}	
	Recalling that $\max(|\theta_m|, |\theta_\infty|)\le \theta_c\leq 1$ and that $\|f_m-f_\infty\|_1\le 2$ and setting $C':=\sup_{|x|\le 4k+4k^2}\Big|\frac{e^x-1}{x}\Big|$, we then have
	\[-2C'kf_\infty(z)\Big(|\theta_m-\theta_\infty|+k|\theta_\infty|\cdot\|f_m-f_\infty\|_1\Big)\le f_m(z)-f_\infty(z)\le 2C'kf_\infty(z)\Big(|\theta_m-\theta_\infty|+k|\theta_\infty|\cdot\|f_m-f_\infty\|_1\Big).\]
	Let $\theta_\infty$ be such that $2k^2C'|\theta_\infty|<1$. Then the above display implies	$||f_m-f_{\infty}||_{1}\to 0$ as $\theta_m\to \theta_\infty$. This completes the proof. Note that the proof of continuity works if $|\theta|\le \theta_c':=\min\Big(1, \frac{1}{2k^2C'},\theta_c\Big)$, which again depends only on $k$. Thus without loss generality one can replace $\theta_c$ by $\theta_c'$ if needed.
	
	\medskip
	
	$(v)$ Let $(\mu_m)_{m\ge 1}$ be a sequence in $\cM$ and $\mu_{\infty}\in\cM$ such that 
	\begin{equation}\label{eq:tv_dist}
		b_m:=\operatorname{TV}(\mu_m,\mu_{\infty})\to 0.
	\end{equation}
	By part $(i)$, we know that there exists $\theta_c>0$ depending only on the size $k$ of $\sigma$ such that, for fixed $\theta\in(-\theta_c,\theta_c)$, all the optimization problems 
	\begin{align*}
		\sup_{\nu\in \cM} [\theta t_{\sigma}(\nu)-D(\nu|\mu_m)]\quad \text{and} \quad \sup_{\nu\in \cM} [\theta t_{\sigma}(\nu)-D(\nu|\mu_{\infty})]
	\end{align*}
	have a unique optimizer $\nu_m:=\nu_{\sigma,\mu_m,\theta}$ and $\nu_\infty:=\nu_{\sigma,\mu_{\infty},\theta}$. We show that $\nu_m\stackrel{w}{\to}\nu_\infty$.
	
	From the assumption in \eqref{eq:tv_dist}, there exists a coupling of $(\mu_m,\mu_{\infty})$ such that for $(Z^{(m)},Z^{(\infty)})\sim (\mu_m,\mu_{\infty})$ we have
	$$\Pr(Z^{(m)}\neq Z^{(\infty)})=b_m.$$	
	Let $({\bf Z^{(m)}},{\bf Z^{(\infty)}})=\big((Z^{(m)}_i,Z^{(\infty)}_i)\big)_{1\le i\le n}\stackrel{iid}{\sim}(Z^{(m)},Z^{(\infty)})$.
	Set 
	$$I_{n,m}:=\left|\{1\le i\le n:Z^{(m)}_i\ne Z^{(\infty)}_i\}\right|\sim \operatorname{Bin}(n, b_m),$$ 
	and use standard concentration bounds (such as Chernoff's inequality) to get that for every $\varepsilon>0$ we have
	\begin{align}\label{eq:ignore}
		\limsup_{m\to\infty}\limsup_{n\to\infty}\frac{1}{n}\log \Pr(I_{n,m}>n\varepsilon)=-\infty.
	\end{align}
	Also, recalling that $\pi_{n,\mu_m}=\pi_{{\bf Z^{(m)}}}$ and $\pi_{n,\mu_\infty}=\pi_{{\bf Z^{(\infty)}}}$, we have 
	$$I_{n,m}\le n \varepsilon\Rightarrow |t_{\sigma}(\pi_{n,\mu_m})-t_{\sigma}(\pi_{n,\mu_\infty})|\le k\varepsilon.$$ 
	This gives	
	\begin{equation*}
		\Ex \exp\Big(n\theta t_{\sigma}(\pi_{n,\mu_m})\Big)  \le e^{n\theta k \varepsilon}\Ex  \exp\Big(n\theta t_{\sigma}(\pi_{n,\mu_\infty})\Big) 
		+  e^{n\theta}\Pr(I_{n,m}>n\varepsilon).
	\end{equation*}
	Thus thanks to  \eqref{eq:ignore}, taking $\log$, dividing by $n$, and letting $n\to\infty$ followed by $m\to\infty$ gives 
	\begin{align*}
		\limsup_{m\to\infty} \lim_{n\to\infty}\frac1n\log \Ex \exp\left(n\theta t_{\sigma}(\pi_{n,\mu_m})\right)\le  \lim_{n\to\infty}\frac1n\log \Ex \exp\left(n\theta t_{\sigma}(\pi_{n,\mu_\infty})\right)+\theta k\varepsilon.
	\end{align*}
	Reversing the roles of $({\bf Z^{(m)}},{\bf Z^{(\infty)}})$, and noting that $\varepsilon>0$ is arbitrary, we get
	\begin{align*}
		\limsup_{m\to\infty} \lim_{n\to\infty}\frac1n\log \Ex \exp\left(n\theta t_{\sigma}(\pi_{n,\mu_m})\right)
		=
		\lim_{n\to\infty}\frac1n\log \Ex \exp\left(n\theta t_{\sigma}(\pi_{n,\mu_\infty})\right).
	\end{align*}
	Recalling the definition in \eqref{eq:gibbs_log_part} and using part $(i)$ of Theorem \ref{thm:gibbs_random}, we get
	\begin{align}\label{eq:bound_above}
		\lim_{m\to \infty} \sup_{\nu\in \cM} [\theta t_{\sigma}(\nu)-D(\nu|\mu_m)]=\sup_{\nu\in \cM} [\theta t_{\sigma}(\nu)-D(\nu|\mu_{\infty})].
	\end{align}
	Recall that $\nu_\infty$ is the unique optimizer of the right-hand side of the above equation, and $(\nu_m)_{m\ge 1}$ is the sequence of unique optimizers of the left-hand side of the above equation. Let $\nu_{\infty}^*$ be a subsequential weak limit of $(\nu_m)_{m\ge 1}$.  Abusing notation slightly, we pass to this subsequence and still denote it by $(\nu_m)_{m\ge 1}$. Since $\nu_m$ is an optimizer, we have
	\[\theta t_\sigma(\nu_m)-D(\nu_m|\mu_m)\ge \theta t_\sigma(\mu_m)\Rightarrow D(\nu_m|\mu_m)\le \theta t_\sigma(\nu_m)-\theta t_\sigma(\mu_m)\le |\theta|.\]
	Lower semi-continuity then gives $D(\nu^*_{\infty}|\mu_{\infty})\le |\theta|$, which implies $\nu^*_{\infty}\in \cM$, as $\mu_\infty\in \cM$. Thus by Lemma \ref{lem:rate} we have  $\lim_{m\to\infty} t_\sigma(\nu_m)= t_{\sigma}(\nu^*_{\infty})$. By using \eqref{eq:bound_above}, we have
	\begin{align*}
		\theta t_{\sigma}(\nu_\infty)-D(\nu_\infty|\mu_{\infty}) = \lim_{m\to\infty} \big(\theta t_{\sigma}(\nu_m)-D(\nu_m|\mu_m)\big) \le \theta t_{\sigma}(\nu^*_{\infty})-D(\nu^*_{\infty}|\mu_{\infty}),
	\end{align*}
	where the inequality uses lower semi-continuity of $D(\cdot|\cdot)$.
	But by uniqueness for the optimizer we must have $\nu^*_{\infty}=\nu_\infty$. Since any subsequential limit is the same, we have shown that $\nu_m\stackrel{w}{\to}\nu_\infty$, as desired.
\end{proof}

\begin{proof}[Proof of Corollary \ref{cor:exp}] 
	$(i)$ To begin, use \eqref{eq:gibbs_lebesgue_gen} with $\mu=\lambda$ to note that
	\begin{align}\label{eq:gibbs_lebesgue2}
		\P(\pi_{n,\sigma,\lambda,\theta}=\tau)= \frac{1}{n!}e^{n\theta t_\sigma(\tau)-F_n(\sigma,\lambda,\theta)} 
	\end{align}
	which verifies \eqref{eq:gibbs_lebesgue}. 
	From Theorem \ref{thm:gibbs_random} part $(iii)$, it follows that $\pi_{n,\sigma,\lambda,\theta}$ satisfies an LDP with speed $n$ and the good rate function $I_{\sigma,\lambda,\theta}(\cdot)$, which for  $\gamma \in \widetilde{\cM}$ equals
	\begin{align*}
		I_{\sigma,\lambda,\theta}(\gamma)=&\inf_{\nu\in \cM:\mathcal{O}(\nu)=\gamma}\{D(\nu|\lambda)-\theta t_\sigma(\nu)\}-\inf_{\nu\in \cM}\{D(\nu|\lambda)-\theta t_\sigma(\nu)\}\\
		=&D(\gamma|\lambda)-\theta t_\sigma(\gamma)-\inf_{\nu\in \widetilde{\cM}}\{D(\nu|\lambda)-\theta t_\sigma(\nu)\},
	\end{align*}
	where the last equality uses Proposition 4.2 in \cite{starr2018phase}. 
	
	\medskip
	
	$(ii)$ Comparing \eqref{eq:gibbs_lebesgue2} along with \eqref{eq:gibbs_lebesgue} we get $$Z_n(\sigma,\theta)=\log n!+F_n(\sigma,\lambda,\theta).$$ 
	From this, the desired conclusion follows on noting the asymptotics of $F_n(\sigma,\lambda,\theta)$ in Theorem \ref{thm:gibbs_random} part $(i)$, along with the observation
	\begin{equation}\label{eq:rew}
		\sup_{\nu\in \cM}\{\theta t_\sigma(\nu)-D(\nu|\lambda)\}=\sup_{\nu\in \widetilde{\cM}}\{\theta t_\sigma(\nu)-D(\nu|\lambda)\}.
	\end{equation}
	Here the above display again uses Proposition 4.2 in \cite{starr2018phase}.
	
	\medskip
	
	$(iii)$ The uniqueness of the optimizer $\nu_{\sigma,\lambda,\theta}$ and the convergence in probability of $\pi_{n,\sigma,\lambda,\theta}$ to $\nu_{\sigma,\lambda,\theta}$ follows from Theorem \ref{thm:gibbs_random2} part $(i)$ and \eqref{eq:rew}.
\end{proof}

We conclude this section showing that for \cnc\ measures $\mu\in\cM$ (recall Definition \ref{def:cc_cnc}), the map $\theta\mapsto \nu_{\sigma,\mu,\theta}$ is indeed non-constant in a small interval around the origin, proving Proposition \ref{prop:cnc}.

\begin{proof}[Proof of Proposition \ref{prop:cnc}]
	For both conclusions, we give the proof for $\theta\in[0,\theta_c]$, noting that the proof for $\theta\in(-\theta_c,0)$ proceeds in a similar manner. Also, we omit the dependence of $(\sigma,\mu)$ on $\nu_{\sigma,\mu,\theta}$, and refer to it as simply $\nu_\theta$ for simplicity of notation.
	
	$(i)$ Suppose $\nu_{\theta_1}=\mu$ for some $\theta_1\in (0,\theta_c)$. Then for any $\theta\in (0,\theta_1)$, part $(ii)$ of Theorem \ref{thm:gibbs_random2} gives
	$$t_{\sigma}(\mu)=t_\sigma(\nu_0)\le t_{\sigma}(\nu_\theta)\le t_{\sigma}(\nu_{\theta_1})=t_{\sigma}(\mu),$$ 
	and so $t_{\sigma}(\nu_\theta)$ is constant on $[0,\theta_1]$. Part $(iii)$ of Theorem \ref{thm:gibbs_random2} then gives that $\nu_\theta\equiv \mu$ for $\theta\in [0,\theta_1]$. Thus in the Euler-Lagrange equation in part $(i)$ of Theorem \ref{thm:gibbs_random}, $g(\cdot)\equiv 1$ for $\theta\in [0,\theta_1]$. Taking $\log$  and differentiating with respect to $\theta$ we get the equation
	\begin{align}\label{teq}
		W(z)=\frac{\int_{[0,1]^2} W(z)e^{\theta W(z)}\d\mu(z)}{\int_{[0,1]^2} e^{\theta W(z)}\d\mu(z)}, 
		\quad\text{with}\quad
		W(z):=k\int_{[0,1]^{2k-2}} h_\sigma(z,z_2,\ldots,z_k)\prod_{a=2}^k \d\mu(z_i).
	\end{align}	
	Note that the right-hand side of the first equation in \eqref{teq} is independent of $z$. Thus $W(z)$ is a constant $\mu$-almost surely, which is equivalent (by the relation in the left-hand side of \eqref{teq}) to saying that $\mu$ is \cc. This is a contradiction.
	
	\medskip
	
	$(ii)$ By part $(i)$ above we have $\nu_\theta \neq \mu$ for all $\theta\in (\theta,\theta_c)$. By Theorem \ref{thm:gibbs_random2} parts $(ii)$ and $(iii)$ we get $t_{\sigma}(\nu_\theta)>t_{\sigma}(\mu)$, as desired.
\end{proof}

\section{Proofs of the applications to constrained $\mu$-random permutations}\label{sec4}

In this section we prove results related to $\mu$-random permutations under constraints.

\begin{proof}[Proof of Theorem \ref{thm:constrain}]
	
	$(i)$ Assume for the moment that the condition of Lemma \ref{lem:ldp} part $(v)$ is satisfied. Using the LDP for $\pi_{n,\mu}$ from Theorem \ref{thm:main}, it follows from Lemma \ref{lem:ldp} part $(i)$ and $(v)$ that conditioned on the event  $\{t_\sigma(\pi_n)\ge \delta\}$, the sequence $\pi_{n,\mu}$ converges in probability to the minimizers of
	$$\inf_{\gamma\in \widetilde{\mathcal{M}}:t_\sigma(\gamma)\ge \delta} I_{\mu}(\gamma)=\inf_{\gamma\in \widetilde{\mathcal{M}}:t_\sigma(\gamma)\ge \delta}\left\{ \inf_{\nu\in \mathcal{M}:\mathcal{O}(\nu)=\gamma}D(\nu|\mu)\right\},$$
	where the last expression follows from the definition of $I_\mu(\cdot)$ given in \eqref{eq:rate}.
	The minimizers of the above optimization problem are easily seen to be $\mathcal{O}(\mathcal{G}(\sigma, \mu,\delta))$, and so we have shown $d(\pi_{n,\mu},\mathcal{O}(\mathcal{G}(\sigma, \mu,\delta)))\xrightarrow{P}0$, as desired. 
	
	To complete the proof, we need to verify the condition of Lemma \ref{lem:ldp} part (v), which translates to $$\inf_{\gamma\in \widetilde{\cM}:t_\sigma(\gamma)\ge \delta}I_\mu(\gamma)=\inf_{\gamma\in \widetilde{\cM}:t_\sigma(\gamma)>\delta}I_\mu(\gamma)\Leftrightarrow \inf_{\nu\in \cM:t_\sigma(\nu)\ge \delta}D(\nu|\mu)=\inf_{\nu\in \cM:t_\sigma(\nu)> \delta}D(\nu|\mu).$$
	The last equality, by \eqref{eq:f}, is the same as $G(\sigma,\mu,\delta)=\lim_{t\rightarrow\delta+}G(\sigma,\mu,t)$, which is equivalent to the assumed right continuity of $G(\sigma,\mu,\cdot)$ at $\delta$.
	
	\medskip
	
	For the rest of the proof we assume that $\mu$ is \cnc\ with respect to $\sigma$ (recall Definition \ref{def:cc_cnc}). Let also $\theta_c$ and $\nu_{\sigma,\mu,\theta}$ be defined as in Theorem \ref{thm:gibbs_random2}. We also set $\nu_{\theta}:=\nu_{\sigma,\mu,\theta}$ for simplicity of notation.
	
	\medskip 
	
	$(ii,a)$ By Theorem \ref{thm:gibbs_random2} part $(ii)$ the map $\theta\mapsto t_{\sigma}(\nu_\theta)$ is continuous and non-decreasing on $(-\theta_c,\theta_c)$. With $\delta_c$ as in the statement of the theorem, i.e., 
	$\delta_c=\sup_{\theta\in (-\theta_c,\theta_c)} t_{\sigma}(\nu_\theta)=\lim_{\theta\to \theta_c^-}t_{\sigma}(\nu_\theta),$
	part $(ii)$ of Proposition \ref{prop:cnc} implies $\delta_c>t_\sigma(\mu)$. The fact that $\hat{\theta}(\delta)$ is well defined for $\delta\in(t_\sigma(\mu),\delta_c)$ and satisfies $t_\sigma\left(\nu_{\hat{\theta}(\delta)}\right)=\delta$ follows from continuity and monotonicity of $\theta\mapsto t_\sigma\left(\nu_{\theta}\right)$.
	
	\medskip
	
	$(ii,b)$ It is enough to show that  for all $\mu'\in\mathcal M$ such that $t_{\sigma}(\mu')\ge \delta$ and $\mu'\neq \nu_{\hat\theta(\delta)}$ we have $D(\mu'|\mu) > D\left(\nu_{\hat\theta(\delta)}\middle|\mu\right)$. Assume on the contrary that there exists $\mu'\neq \nu_{\hat\theta(\delta)}$ such that $D(\mu'|\mu)\le D\left(\nu_{\hat\theta(\delta)}\middle|\mu\right)$ and $t_{\sigma}(\mu')\ge \delta$, then
	\begin{align*}
		\hat\theta(\delta) t_{\sigma}(\mu')-D(\mu'|\mu) \ge \hat\theta(\delta) t_{\sigma}\left(\nu_{\hat\theta(\delta)}\right)-D\left(\nu_{\hat\theta(\delta)}\middle|\mu\right),
	\end{align*} 
	where we used that $t_\sigma\left(\nu_{\hat{\theta}(\delta)}\right)=\delta$ from part $(ii,a)$.
	As $\nu_{\hat\theta(\delta)}$ is the optimizer of the right-hand side of \eqref{eq:opt}, we must have $\mu'$ to be another optimizer. But this contradicts the uniqueness of the maximizer of the right-hand side of \eqref{eq:opt} (part $(i)$ of Theorem \ref{thm:gibbs_random2} used with $\hat{\theta}(\delta)\in (0,\theta_c)$). Thus $\mathcal{G}(\sigma,\mu,\delta)=\left\{\nu_{\hat\theta(\delta)}\right\}$, as desired.  
	
	\medskip
	
	$(ii,c)$ This is a consequence of Theorem \ref{thm:gibbs_random} part $(ii)$. 
	
	\medskip
	
	$(ii,d)$ It suffices to show that $G(\sigma,\mu,\cdot)$ is right continuous on $(t_{\sigma}(\mu),\delta_c)$; then the result follows from part $(i)$. To this effect, let $\delta_k \downarrow \delta \in (t_{\sigma}(\mu),\delta_c)$. By definition of $\hat{\theta}(\delta)\in(0,\theta_c)$ as in the statement of the theorem, we have $\hat{\theta}(\delta_k) \downarrow \theta'$ for some $\theta'\in [0,\theta_c]$. We claim that $\theta'=\hat\theta(\delta)$. Since $\theta\mapsto t_{\sigma}(\nu_{\theta})$ is non-decreasing and continuous for $\theta\in(-\theta_c,\theta_c)$ (Theorem \ref{thm:gibbs_random2} parts $(ii)$), we have 
	$$t_{\sigma}(\nu_{\theta'})= \lim_{k\to\infty}t_\sigma\left(\nu_{\hat\theta(\delta_k)}\right)=\lim_{k\to\infty}\delta_k=\delta,$$ 
	where we used again that $t_\sigma\left(\nu_{\hat\theta(\delta_k)}\right)=\delta_k$ from $(ii,a)$.
	Thus by definition $\hat\theta(\delta) \ge \theta'$. On the other hand, form Proposition \ref{prop:cnc} part $(ii)$ we have that $t_{\sigma}\left(\nu_{\hat\theta(\delta_k)}\right)>t_{\sigma}\left(\nu_{\hat\theta(\delta)}\right)$, which forces $\hat\theta(\delta_k) > \hat\theta(\delta)$. Taking $k\to \infty$, we see that $\theta' \ge \hat\theta(\delta)$. Hence $\theta'=\hat\theta(\delta)$. 
	
	Recall now the definition of $F(\sigma,\mu,\theta)$ from \eqref{eq:opt}. 
	By continuity of $F(\sigma,\mu,\cdot)$ proved in Theorem \ref{thm:gibbs_random2} parts $(ii)$ and the continuity of $t_{\sigma}(\nu_{\cdot})$ on $(-\theta_c,\theta_c)$ we have
	\begin{align*}
		G(\sigma,\mu,\delta_k)=D\left(\nu_{\hat\theta(\delta_k)}\middle|\mu\right) & \stackrel{\eqref{eq:opt}}{=}{\hat\theta(\delta_k)} t_{\sigma}\left(\nu_{\hat\theta(\delta_k)}\right)-F\left(\sigma,\mu,{\hat\theta(\delta_k)}\right) \\ & \to {\hat\theta(\delta)} t_{\sigma}\left(\nu_{\hat\theta(\delta)}\right)-F\left(\sigma,\mu,{\hat\theta(\delta)}\right)=G(\sigma,\mu,\delta).
	\end{align*}
	Thus $G(\sigma,\mu,\cdot)$ is right continuous. This completes the proof.
\end{proof}

\section{Proofs of the applications to inversions and examples of \cc\ permutons with respect to inversions} \label{sec5}
We prove here all the results that focus on the specific case of inversions. We start by proving in the next section Proposition \ref{prop:counter} and \ref{prop:xi}, concerning the non-uniqueness of the optimizers for the optimization problems appearing in Theorem \ref{thm:gibbs_random2} and Theorem \ref{thm:constrain}.

\subsection{Non-uniqueness of the optimizers}
We recall that $D_{11}=[0,\frac12]^2,$ $D_{21}=[\frac12,1]\times[0,\frac12],$ $D_{21}=[0,\frac12]\times [\frac12,1]$ and $D_{22}=[\frac12,1]^2$.

\begin{proof}[Proof of Proposition \ref{prop:counter}] 
	Given any measure $\nu\in \cM$ such that $D(\nu|\xi)<\infty$, we can identify two functions $f,g:[0,\frac12]\to \R$ such that 
	$$\frac{\d\nu}{\d\xi}(x,y)= f(x)\cdot\mathds{1}{\left\{x\in [0,\tfrac12),y=\tfrac12-x\right\}}+ g(x-\tfrac12)\cdot\mathds{1}{\left\{x\in [\tfrac12,1],y=\tfrac32-x\right\}}.$$ 
	Let $\lambda$ be the Lebesgue measure on $[0,1]$. Then note that $$\int_{[0,1/2]} (f+g)\d \lambda=1 \qquad \text{ and }\qquad t_{{21}}(\nu)=\left(\int_{[0,1/2]} f\d \lambda\right)^2+\left(\int_{[0,1/2]}g \d \lambda\right)^2=p^2+q^2,$$
	where $p=\int_{[0,1/2]}fd\lambda=\nu(D_{11}), q=\int_{[0,1/2]}gd\lambda=\nu(D_{22})$, with $p+q=1$. Furthermore,
	\begin{align*}
		D(\nu|\xi)=&\int_{[0,1/2]} f\log f\d \lambda+\int_{[0,1/2]} g \log g\d \lambda\\
		=&p\log p+q\log q+\log 2+pD(\gamma_f|2\lambda)+qD(\gamma_g|2\lambda),
	\end{align*}
	where $\gamma_f$ and $\gamma_g$ are the probability measures on $[0,1/2]$ induced by the functions $f,g$ respectively, and $2\lambda$ is the Lebesgue measure on $[0,1/2]$.
	Thus we have (recall that $\mathcal{F}({21},\xi,\theta)$ denotes the set of optimizers for the optimization problem in the right-hand side of \eqref{eq:opt})
	\begin{align*}
		\mathcal{F}({21},\xi,\theta)=&\sup_{\nu\in \cM}\{\theta t_{21}(\nu)-D(\nu|\xi)\}\\
		=&\sup_{p, f, g}\{\theta (p^2+q^2)-p\log p-q\log q-pD(\gamma_f|2\lambda)-qD(\gamma_g|2\lambda)\}.
	\end{align*}
	The optimum is attained when
	$\gamma_f=\gamma_g=2\lambda$, which implies the distribution restricted to the boxes $D_{11}, D_{22}$ are uniform (denoted by $\xi_{11},\xi_{22}$ respectively, in the statement of the proposition). To characterize the optimizers completely,  it remains to solve the optimization problem
	$$\sup_{p\in [0,1]}\{\theta(p^2+q^2)-p\log p-q\log q\},$$
	which on setting $x=p-q$ becomes
	$$\sup_{x\in [-1,1]}\left\{\frac{\theta}{2}(1+x^2)-\frac{1+x}{2}\log \frac{1+x}{2}-\frac{1-x}{2}\log\frac{1-x}{2}\right\}.$$
	This is the same optimization problem obtained from analyzing the partition function of the Curie--Weiss--Ising model (see Chapter 2 of \cite{vel} for example). The optimizers are given by
	\begin{align*}
		x_{\rm opt}=\begin{cases}
			0\text{ if }\theta\le 1,\\
			\{-m_\theta,m_\theta\}\text{ if }\theta>1.
		\end{cases}
	\end{align*} 
	Here $m_\theta$ is as defined in the statement of the proposition. We have thus explicitly computed the optimizers for all $\theta\in\R$, and proved that $|\mathcal{F}({21},\xi,\theta)|=2$ when $\theta>1$. To conclude that $|\mathcal{O}(\mathcal{F}({21},\xi,\theta))|=2$, we invoke Lemma \ref{lem:dmat} part $(ii)$. To show that the $\mathcal{O}(\cdot)$ projections of the two optimizers in $\mathcal{F}({21},\xi,\theta)$ are distinct, it suffices to note that for both the optimizers $\nu$ in $\mathcal{F}({21},\xi,\theta)$ we have 
	$$\nu(D_{11})\ne \nu(D_{22})\quad \text{ and } \quad \nu(D_{12}\cup D_{21})=0, $$
	and so the assumption in \eqref{eq:dmat} holds.
\end{proof}

\begin{proof}[Proof of Proposition \ref{prop:xi}]
	Note that by definition of $\xi$, we have that $t_{21}(\xi)=\frac12$ and $\alpha_{21}(\xi)=1$. With a simple computation, we also have that $\xi$ is a \cc\ permuton. 
	
	It remains to investigate the set of optimizers $\mathcal{G}(21,\xi,\delta)$ of the optimization problem in \eqref{eq:f}. 
	Fix $\delta\in (\frac12,1)$. Going through a similar calculation as in the proof of Proposition \ref{prop:counter}, 
	we need to optimize
	$$\inf_{p,f,g: p^2+q^2\ge \delta}\{p\log p+(1-p)\log (1-p)+\log 2+pD(\gamma_f|2\lambda)+qD(\gamma_g|2\lambda)\}.$$
	Thus the minimization problem in \eqref{eq:f} is attained when
	$\gamma_f=\gamma_g=2\lambda$.
	We thus need to optimize
	\begin{align*}
		G(21,\xi,\delta)=\inf_{p^2+(1-p)^2 \ge \delta} \Big[p\log p+(1-p)\log (1-p)+\log 2\Big],
	\end{align*}
	which has two solution $\frac{1\pm \sqrt{2\delta-1}}{2}$. Consequently the minimizers are given by 
	$$\left\{\frac{1+\sqrt{2\delta-1}}{2}\xi_{11}+\frac{1-\sqrt{2\delta-1}}{2}\xi_{22}, \frac{1-\sqrt{2\delta-1}}{2}\xi_{11}+\frac{1+\sqrt{2\delta-1}}{2}\xi_{22}\right\},$$
	where we recall that $\xi_{11}$ and $\xi_{22}$ are the uniform measures on the diagonal boxes $D_{11}$ and $D_{22}$, respectively. 
	Hence $|\mathcal{G}(21,\xi,\delta)|=2$. To conclude that $|\mathcal{O}(\mathcal{G}(21,\xi,\delta))|=2$, we can proceed as in the final part of the previous proof.
\end{proof}

\subsection{Interchanging conditioning events and limits}

We now prove another application to inversions, that is Theorem \ref{thm:gibbs4}.

\begin{proof}[Proof of Theorem  \ref{thm:gibbs4}] 
	$(i)$ We start by assuming that all the assumptions of Lemma \ref{lem:ldp} parts $(i)$ and $(v)$ hold for
	$$X_n=\pi_{n,\sigma,\lambda,\theta},\quad J(\gamma)=I_{\sigma,\theta}(\gamma), \quad T(\cdot)=t_\sigma(\cdot),\quad U=(-\infty,\tfrac{1}{k!}),$$
	where we recall that $I_{\sigma,\theta}(\gamma)$ was defined in \eqref{eq:gd_rate}.
	Then the desired conclusion follows on noting that the unique solution to the optimization problem
	\begin{align*}
		\inf_{\nu\in \widetilde{\cM}:t_\sigma(\nu)\le \frac{1}{k!}}\{D(\nu|\lambda)-\theta t_\sigma(\nu)\}
	\end{align*}
	is $\lambda$ (the Lebesgue measure on $[0,1]^2$).
	
	Note that the general assumption stated at the beginning of Lemma \ref{lem:ldp} hold thanks to Corollary \ref{cor:exp}. Hence, to complete the proof, we are left to show that part $(v)$ of Lemma \ref{lem:ldp} is applicable, for which we need to verify that
	$$\inf_{\nu\in \cM:t_\sigma(\nu)\le \frac{1}{k!}} \{D(\nu|\lambda)-\theta t_\sigma(\nu)\}=\inf_{\nu\in \cM:t_\sigma(\nu)<\frac{1}{k!}} \{D(\nu|\lambda)-\theta t_\sigma(\nu)\}.$$
	As shown above, the left-hand side has the unique optimizer $\lambda$. To complete the proof, it thus suffices to show that
	\begin{align}\label{eq:effect}
		\inf_{\nu\in \cM:t_\sigma(\nu)<\frac{1}{k!}} \{D(\nu|\lambda)-\theta t_\sigma(\nu)\}\le -\tfrac{\theta}{k!}.
	\end{align}
	To this effect, setting $\nu_{\alpha}:=\nu_{\sigma,\lambda,\alpha}$ as in Theorem \ref{thm:gibbs_random2} and recalling that from Theorem \ref{thm:gibbs_random2} part $(iv)$ the map $\alpha\mapsto \nu_\alpha$ is continuous in total variation for $\alpha\in (-\theta_c,\theta_c)$, we get
	$$ \nu_\alpha\stackrel{w}{\to}\lambda \quad\mbox{as}\quad \alpha \to 0^-,\qquad\text{ and }\qquad t_\sigma(\nu_\alpha)<t_\sigma(\lambda)=\tfrac{1}{k!},\quad \text{for} \quad\alpha\in (-\theta_c,0),$$
	where the second claim follows from Proposition \ref{prop:cnc} part $(ii)$.
	Thus as $\alpha\to 0-$, we have
	\begin{align*}
		\theta t_\sigma(\nu_\alpha)-D(\nu_\alpha|\lambda)&=(\theta-\alpha)t_\sigma(\nu_\alpha)+\alpha t_\sigma(\nu_\alpha)-D(\nu_\alpha|\lambda)\\
		&=(\theta-\alpha)t_\sigma(\nu_\alpha)+\sup_{\nu\in \cM}\{\alpha t_\sigma(\nu)-D(\nu|\lambda)\}\\ &
		\rightarrow \theta t_\sigma(\lambda)+\sup_{\nu\in \cM}\{-D(\nu|\lambda)\}=\tfrac{\theta}{k!},
	\end{align*}
	where in the last line we used the continuity of $t_\sigma(\cdot)$ at $\lambda$ (Lemma \ref{lem:rate}) and  the continuity in $\alpha$ of $F(\sigma,\lambda,\alpha)=\sup_{\nu\in \cM}\{\alpha t_\sigma(\nu)-D(\nu|\lambda)\}$ (Theorem \ref{thm:gibbs_random2} part $(ii)$). This verifies \eqref{eq:effect}, and hence completes the proof of part $(i)$ of the theorem.
	
	\medskip
	
	$(ii)$ Recall that $\nu_{\theta}=\nu_{21,\lambda,\theta}$. By \cite{Shannon,starr2018phase} we know $\nu_\theta$ has a density supported on all of $[0,1]^2$. Hence by Proposition \ref{prop:CNC21}, $\nu_{\theta}$ is \cnc. Thus by Theorem \ref{thm:constrain} $(ii)$ (and Remark \ref{rem:o_directict}), $\pi_{n,\nu_{\theta}}$ conditioned on the event $\{t_{{21}}(\pi_{n,\nu_{\theta}})\le \frac1{2}\}$ converges weakly to $\mathcal{O}(\nu_{{21},\nu_{\theta},\beta})$, for some $\beta<0$ (depending on $\theta$) satisfying $t_{21}(\nu_{{21},\nu_{\theta},\beta})=\frac{1}{2},$ provided that
	\begin{align}\label{eq:check}
		t_{21}(\nu_{{21},\nu_{\theta},-\theta_c})<\tfrac{1}{2}.
	\end{align}
	
	To complete the proof, we need to verify that \eqref{eq:check} holds for all $\theta$ small enough, and also verify that the measure $\mathcal{O}(\nu_{{21},\nu_{\theta},\beta})$ is not Lebesgue measure, or equivalently, the measure $\nu_{{21},\nu_{\theta},\beta}$ is not a product measure. Deferring the proof of \eqref{eq:check}, we first verify the second conclusion.
	
	To this end, setting $g:=\frac{d\nu_{{21},\nu_{\theta},\beta}}{d\nu_{\theta}}$, using the fixed point equation of Theorem \ref{thm:constrain} part $(ii,c)$ yields
	\[ g(x,y)=\Con^{-1} \exp\Big(2\beta \int_{[0,1]^{2}}h_{21}((x,y),(u,v)) g(u,v) d\nu_{\theta}(u,v)\Big).\]
	Let $f$ be the density of $\nu_{\theta}$ w.r.t.~Lebesgue measure. Then $\nu_{{21},\nu_{\theta},\beta}$ has a density $fg$ with respect to Lebesgue measure and the above integral equation becomes
	\[ g(x,y)=\Con^{-1} \exp\Big(2\beta \int_{[0,1]^{2}}h_{21}((x,y),(u,v)) g(u,v)f(u,v)dudv\Big).\] 
	By way of contradiction, assume $\nu_{{21},\nu_{\theta},\beta}$ is a product measure. Then $f(x,y)g(x,y)\stackrel{a.s.}{=}a(x)b(y)$. Plugging this above gives
	\begin{align*}
		{f(x,y)} & =\Con \cdot a(x)b(y)\exp\Big(-2\beta \int_{[0,1]^{2}}h_{21}((x,y),(u,v)) a(u)b(v)du dv\Big).
	\end{align*}
We use the definition of $h_{\sigma}$ from \eqref{def:hs1} with $\sigma=21$. The expression in \eqref{def:hs1} simplifies to
\begin{align*}
	h_{21}((x,y),(u,v)) = \ind\{x>u,v<y\}+\ind\{x<u,v>y\}.
\end{align*}
Using this we get that
\begin{align*}
	{f(x,y)} 
	=\Con\cdot a(x)b(y)\exp\Big(-2\beta[A(x)(1-B(y))+(1-A(x))B(y))]\Big).
\end{align*}
	where $A(x):=\int_0^x a(u)du, B(y):=\int_0^y b(u)du$. This leads to
	\begin{align*}
		\frac{\partial^2}{\partial x \partial y} \log f(x,y)= 4\beta \cdot a(x)b(y).
	\end{align*}
	But the Euler-Lagrange equation for $f$ forces $\frac{\partial^2}{\partial x \partial y} \log f(x,y)= 4\theta \cdot f(x,y)$ (see also \cite{Shannon,starr2018phase}). This implies $\theta=\beta$ and that $\nu_{\theta}$ is a product measure, which is a contradiction for $\theta\neq 0$.	
	
	To complete the proof, it remains to verify \eqref{eq:check} for all $\theta$ small enough. To this effect, using Theorem \ref{thm:gibbs_random2} $(iv)$ we have $$\nu_{{21},\lambda,\theta}\stackrel{TV}{\to} \nu_{{21},\lambda,0}=\lambda\text{ as }\theta \downarrow 0.$$ This, along with Theorem \ref{thm:gibbs_random2} (v) and Lemma \ref{lem:rate} give
	$$\nu_{{21}, \nu_{\theta},-\theta_c} \stackrel{w}{\to} \nu_{-\theta_c}\Rightarrow t_{{21}}(\nu_{{21}, \nu_{\theta},-\theta_c} )\to t_{{21}}(\nu_{-\theta_c})<\tfrac12.$$ 
	Here the last inequality uses Proposition \ref{prop:cnc}, and the fact that $\lambda$ is \cnc. This completes the proof of the theorem.
\end{proof}

\subsection{Existence of a phase transition for a generalized version of the Mallows model}

This section is devoted to the proof of Theorem \ref{ppn:counter}. Before going to the proof of the theorem, we first establish a technical lemma about the measure $\mu_\ell$ introduced in Definition \ref{def:counter2}. Recall that $D_{11}=[0,\frac12]^2,$ $D_{21}=[\frac12,1]\times[0,\frac12],$ $D_{21}=[0,\frac12]\times [\frac12,1]$ and $D_{22}=[\frac12,1]^2$.

\begin{lemma}\label{lem:inv_symm}
	For any measure $\nu$ with density $\frac{d\nu}{d\lambda}=g$ with respect to Lebesgue measure on the unit square, define a measure $\tilde{\nu}$ on the unit square $[0,1]^2$ with density
	$$\frac{d\tilde{\nu}}{d\lambda}(x,y):=4g(2x,2y)\mathds{1}{\{(x,y)\in D_{11}\}}.$$
	Then the following conclusions hold:
	
	(i) $t_\sigma(\nu)=t_\sigma(\tilde{\nu})$, for any $\sigma\in S$.
	
	(ii) $D(\nu|\mu_\ell)-D(\tilde{\nu}|\mu_\ell)=\nu(D_{12}\cup D_{21})[\log(1+\ell)-\log (1-\ell)]-\log 4.$
\end{lemma}

\begin{proof}
	$(i)$ This is immediate on noting that $\tilde{\nu}$ is the push-forward measure of $\nu$ under the map $(x,y)\mapsto (x/2,y/2)$, which is monotone in each coordinate.
	
	\medskip
	
	$(ii)$  Setting $$g:=\frac{d\nu}{d\lambda},\quad \tilde{g}:=\frac{d\tilde{\nu}}{d\lambda},$$
	a direct computation gives
	\begin{align*}
		D(\nu|\mu_\ell)=&D(\nu|\lambda)-\nu(D_{11}\cup D_{22}) \log(1+\ell)-\nu(D_{12}\cup D_{21})\log(1-\ell),\\
		D(\tilde{\nu}|\mu_\ell)=&D(\tilde{\nu}|\lambda)-\log(1+\ell)=D(\nu|\lambda)+\log 4-\log(1+\ell).
	\end{align*}
	On taking a difference, we get $(ii)$.
\end{proof}

We turn to the proof of the main result of this section.

\begin{proof}[Proof of Theorem \ref{ppn:counter}] 
	$(i)$ Since for all $\ell\in [0,1]$, the support $\supp(\mu_{\ell})$ contains at least one between $D_{11}$ $D_{12}$, the desired conclusion follows from Proposition \ref{prop:CNC21} part $(i)$. Note also that $t_{21}(\mu_\ell)=\frac{2-\ell}{4}$, and $\alpha_{21}(\mu_\ell)=1$ for all $\ell\in [0,1]$.
	
	\medskip
	
	$(ii)$ If $\ell=0$ we have by definition $\mu_0=\lambda$, the Lebesgue measure on $[0,1]^2$. In this case, with $\sigma=21$, the optimizer $\nu_{\theta}$ for part $(i)$ of Theorem \ref{thm:gibbs_random} is known to be unique (see \cite{Shannon,starr2018phase}) for all $\theta\in\R$ and its density is given by 
	\begin{align}\label{smm}
		\Phi_\theta(x,y)=\frac{\frac{\theta}{2}\sinh\frac{\theta}{2}}{\left[e^{-\frac{\theta}4}\cosh(\theta(x-y)/2)-e^{\frac\theta4}\cosh(\theta(x+y-1)/2)\right]^2}.
	\end{align}
	Thus for each $\delta\in (\frac12,1)$ there exists a unique $\theta>0$ such that  $t_{{21}}(\nu_\theta)=\delta$. Following the arguments in part $(ii)$ of Theorem \ref{thm:constrain} we see that $\mathcal{G}(21,\lambda,\delta)=\left\{\nu_{\hat{\theta}(\delta)}\right\}$ for all $\delta\in (\frac12,1)$.
	
	\medskip
	
	$(iii)$ Since $\mu_\ell$ is \cnc, by Theorem \ref{thm:constrain} part $(ii)$ there exists $\delta_c(\ell)>t_{21}(\mu_\ell)=\frac{2-\ell}{4}$ such that $\mathcal{G}(21,\mu_\ell,\delta)$ has cardinality $1$ for $\delta\in (\frac{2-\ell}{4},\delta_c(\ell))$.
	
	\medskip
	
	$(iv)$ For $\delta\in (\frac{2-\ell}{4},1)$, let  $\nu\in \mathcal{G}({21}, \mu,\delta)$, and assume without loss of generality that 
	\begin{align}\label{eq:bound0}
		\nu(D_{11})\ge \nu(D_{22}).
	\end{align}
	Since $\nu$ is an optimizer, with $\tilde{\nu}$ as in Lemma \ref{lem:inv_symm} we must have $D(\nu|\mu_\ell)\le D(\tilde{\nu}|\mu_\ell)$, which in turn using Lemma \ref{lem:inv_symm} part $(ii)$ gives
	\begin{align}\label{eq:bound1}
		\nu(D_{12}\cup D_{21})\le \frac{\log 4}{\log(1+\ell)-\log(1-\ell)}.
	\end{align}
	Note also that 
	$$t_{21}(\nu)\le 1-2\nu(D_{11})\nu(D_{22})$$
	since pairs of points in $D_{11}$ and $D_{22}$ do not contribute to inversions.
	Invoking \eqref{eq:bound0} and recalling that $t_{21}(\nu)\ge \delta$, we get from the previous equation that
	\begin{align}\label{eq:bound2}
		\nu(D_{22})\le \sqrt{\frac{1-\delta}{2}}.
	\end{align}
	Note that there exists $\ell_c<1$ such that the following is true. For all $\ell\in(\ell_c,1]$ there exists $\delta'_c(\ell)<1$ such that
	\begin{equation}
		1-2\sqrt{\frac{1-\delta}{2}}-\frac{9\log 4}{\log(1+\ell)-\log(1-\ell)}>0,\quad \text{for all }\delta\in\delta'_c(\ell),1),
	\end{equation}
	Combining this bound with \eqref{eq:bound1} and \eqref{eq:bound2}, it follows that 
	\begin{align}
		\nu(D_{11})-\nu(D_{22})=1-2\cdot\nu(D_{22})-\nu(D_{12}\cup D_{21})>8\cdot\nu(D_{12}\cup D_{21}),
	\end{align} 
	whenever $\ell\in(\ell_c,1]$ and $\delta\in\delta'_c(\ell),1)$.
	In particular, for these parameters the assumption of Lemma \ref{lem:dmat} part $(ii)$ holds. Therefore, we get $\mathcal{O}(\nu)\ne \mathcal{O}(\tilde{\nu})$, where $\tilde{\nu}$ is the push forward measure of $\nu$ via the map $T(x,y)=(1-x,1-y)$. But it is immediate that if $\nu$ is an optimizer, then so is $\tilde{\nu}$. Thus the cardinality of $\mathcal{O}(\mathcal{G}({21}, \mu,\delta))$ is at least $2$.
	\end{proof}

\subsection{Characterization of \cc\ and \cnc\ measures with respect to inversions}\label{sect:21CNCperm}

In this section we prove Proposition \ref{prop:CNC21}, which gives some characterizations of \cnc\ measures with respect to inversions.
We begin with the following lemma.

\begin{lemma}\label{lem:sum_area_is_const}
	Let $\mu$ be a \cc\ measure with respect to $\sigma=21$. Then there exists a constant $c\in[0,1]$ such that for all $(x,y)\in \supp(\mu)$ we have
	$$\mu([0,x]\times [0,y])+\mu([x,1]\times [y,1])=c.$$ 
\end{lemma}

\begin{proof}
	Let $\big((X_i,Y_i)\big)_{i=1,2}\stackrel{i.i.d.}{\sim}\mu$. Recall that $\mu$ is \cc\ with respect to $21$, if there exists a constant $c\in[0,1]$ such that
	$\P(\pi_{{\bf X},{\bf Y}}=21|(X_1,Y_1))=1-c$ almost surely. Since
	$$\P(\pi_{{\bf X},{\bf Y}}=21|(X_1,Y_1))=\mu\left([0,X_1]\times [Y_1,1]\right)+\mu\left([X_1,1]\times [0,Y_1]\right)$$,
	taking the complement, we a.s.\ have that
	$$\mu\left([0,X_1]\times [0,Y_1]\right)+\mu\left([X_1,1]\times [Y_1,1]\right)=c.$$
	Now the function $(x,y)\mapsto\mu\left([0,x]\times [0,y]\right)+\mu\left([x,1]\times [y,1]\right)$ is continuous, as $\mu$ has continuous marginals. Thus the set $\{(x,y):\mu\left([0,x]\times [0,y]\right)+\mu\left([x,1]\times [y,1]\right)=c\}$ is closed, and has $\mu$-measure $1$. Since $\supp(\mu)$ is the smallest closed set with probability $1$, it follows that
	$$\supp(\mu)\subseteq\{(x,y)\in {[0,1]^2}:\mu\left([0,x]\times [0,y]\right)+\mu\left([x,1]\times [y,1]\right)=c\},$$
	as desired.
\end{proof}

We can now prove Proposition \ref{prop:CNC21}.

\begin{proof}[Proof of Proposition \ref{prop:CNC21}]
	Let  $\mu$ be a \cc\ measure with respect to $21$. All subsequent a.s.~statements in this proof will be with respect to $\mu$. By Lemma \ref{lem:sum_area_is_const}, we know that there exists a constant $c\in [0,1]$ such that for all $(x,y)\in \supp(\mu)$ we have
	\begin{equation}\label{eq:equality_area_11}
		\mu\left([0,x]\times [0,y]\right)+\mu\left([x,1]\times [y,1]\right)=c.
	\end{equation}
	
	\medskip
	
	$(i)$ Assume, aiming for a contradiction, that the interior of $\supp(\mu)$ is non-empty.   Then there exists an open rectangle $R\subset \supp(\mu)$. Fix $x_1,x_2,y_1,y_2\in [0,1]$, such that $x_1<x_2$, $y_1<y_2$, and for all $1 \le i,j \le 2$, $(x_i,y_j)\in R$. Using these four points, the square $[0,1]^2$ can be divided into nine rectangles with $\mu$ mass $A$, $B$, $C$, $D$, $E$, $F$, $G$, $H$ and $I$, as indicated in Figure~\ref{fig:nine_areas}. The rest of the proof is devoted to showing that $E=0$. This would be a contradiction because $R\subseteq\supp(\mu)$, so no open subset of $R$ can have measure zero.
	\begin{figure}[ht]
		\begin{center}
			\includegraphics[width=2.0in,height=2.0in]{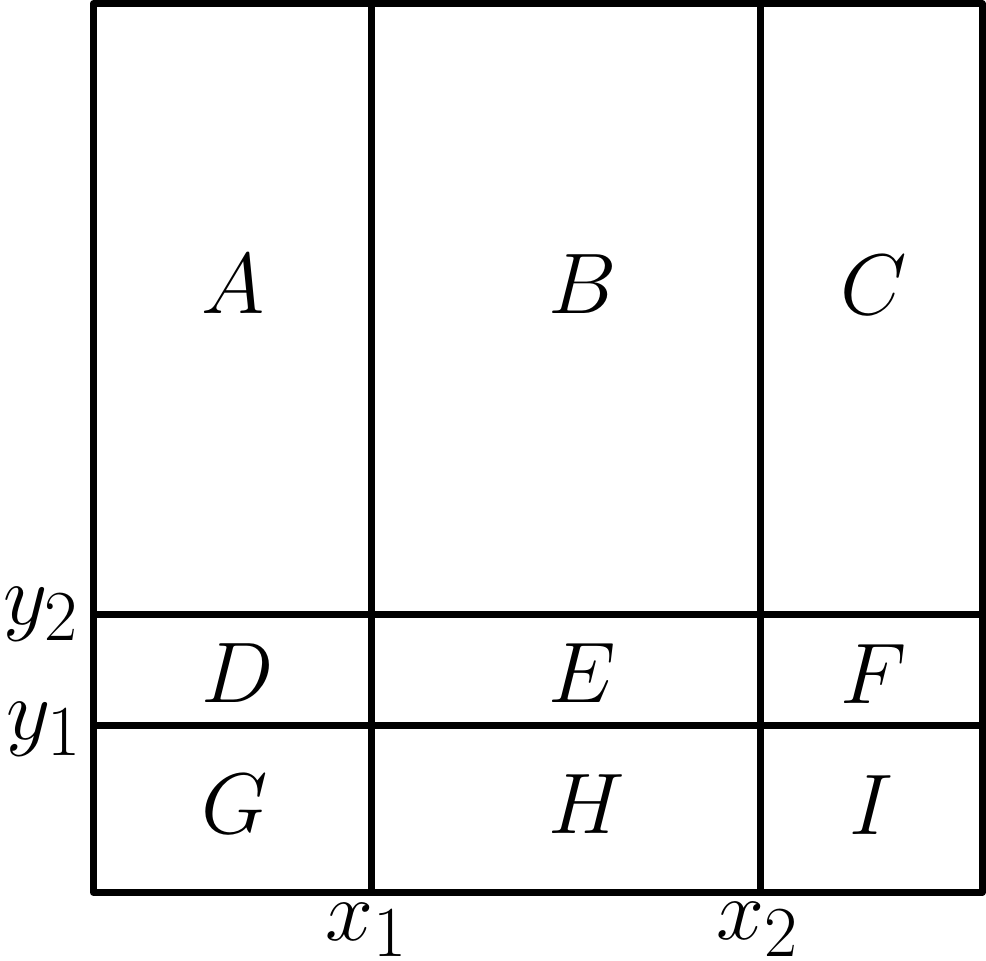}
			\caption{The square $[0,1]^2$ is divided by the points $(x_1,y_1)$ and $(x_2,y_2)$ into nine rectangles with $\mu$-mass $A$, $B$, $C$, $D$, $E$, $F$, $G$, $H$ and $I$, as shown in the picture.}
			\label{fig:nine_areas}
		\end{center}
	\end{figure}
	
	Using \eqref{eq:equality_area_11} with the pair of points $\{(x_1,y_1),(x_2,y_2)\}$, we obtain the equation
	\begin{equation}\label{eq:rel1}
		B+C+E+F+G=C+D+E+G+H \Longrightarrow B=D+H-F.
	\end{equation} 
	Similarly, \eqref{eq:equality_area_11} with the pair of points $\{(x_1,y_2),(x_2,y_2)\}$ and $\{(x_1,y_2), (x_2,y_1)\}$ we get the equations:
	\begin{align}
		B+C+D+G=C+D+E+G+H\Longrightarrow E=B-H,\label{eq:rel2}\\
		B+C+D+G=C+F+G+H\Longrightarrow H=B+D-F.\label{eq:rel3}
	\end{align}
	Substituting \eqref{eq:rel1} and \eqref{eq:rel3} in \eqref{eq:rel2}, we obtain that $E=H-B$. The latter relation and \eqref{eq:rel2} finally give that $E=0$.
	
	\medskip
	
	$(ii)$ Let $\mu$ be a \cc\ permuton with respect to $\sigma = 21$.
	Let $(x,y)\in\supp(\mu)$ be such that $\mu([0,x]\times[0,y])=0$. 
	Fix a second point $(x',y')\in\supp(\mu)$, such that $(x',y')\neq (x,y)$ and $\mu([0,x']\times[0,y'])=0$. 
	From Lemma \ref{lem:sum_area_is_const} we have that
	$$\mu([0,x]\times [y,1])+\mu([x,1]\times [0,y])=\mu([0,x']\times [y',1])+\mu([x',1]\times [0,y']).$$
	Now using that $\mu([0,x]\times[0,y])=0$, $\mu([0,x']\times[0,y'])=0$ and that $\mu$ has uniform marginals, we can rewrite the last equation as
	$$x+y=x'+y',$$
	concluding that the points $(x,y)$ and $(x',y')$ lie on the same line $x + y = b$ of slope $-1$. The constant $b$ is thus uniquely defined.
	
	\medskip
	
	$(iii)$  If $b=0$ the statement is trivial, so we assume $b>0$. Consider the continuous function
	$F(x,y)=\mu([0,x]\times [0,y])$. Assume for the sake of contradiction that $\supp(\mu)\cap \triangle \neq \varnothing$. Since $\triangle$ does not include the line $x+y=b$ by definition, there must exist $(u,v)\in \triangle$ such that $F(u,v)>0$. Set $$\til{b}=\operatorname{arg} \operatorname{inf}\{x+y \mid F(x,y)>0\}.$$
	Observe that $\til{b} \le u+v < b$. Take a sequence $(x_k,y_k)$ with $F(x_k,y_k)>0$ and $x_k+y_k\to \til{b}$. Passing to a subsequence, we may assume $x_k\to x_{\infty}$ and $y_k\to y_{\infty}$. Then by continuity of  $F$, we have $F(x_{\infty},y_{\infty})=0$ and $(x_{\infty},y_{\infty})\in \supp(\mu)$. Then by part $(ii)$, $\til{b}=b$ which is a contradiction.
\end{proof}

\subsection{Examples of  \cc\ permutons with respect to inversions}\label{sect:examples}

We start by constructing a one parameter family $\{\mu^{(z)}\}_{z\in[0,1]}$ of \cc\ permutons  with respect to inversions. Heuristically $\mu^{(z)}$ is the permuton obtained by spreading Lebesgue measure on the unique rectangle with inclination $45$ degrees inscribed in a unit square and having bottom intersection point at $(z,0)$ (see the left-hand side of Figure \ref{fig:rect_perm}). More formally, let $z$ be a point in $[0,1]$.  Let $L_1$ and $L_3$ denote the line segments with slope $-1$ connecting $(0,z)$ to $(z,0)$ and $(1-z,1)$ to $(1,1-z)$, respectively.  Similarly let $L_2$ and $L_4$ denote the line segments with slope $1$ connecting  $(z,0)$ to $(1,1-z)$ and $(0,z)$ to $(1-z,1)$, respectively.  The union of $L_1$, $L_2$, $L_3$ and $L_4$ forms a rectangle $R^{(z)}$ in $[0,1]^2.$  
For each of the line segments $L_i$ ($i=1,2,3,4$) we will define a measure $\mu^{(z)}_i$ as a rescaled Lebesgue measure. Let $\lambda$ be the Lebesgue measure on $[0,1]$.  Let $A$ be a Borel measurable set on $[0,1]^2$.  For each $i$, let $A_i = A\cap L_i$.  Finally let $\pi_x(A_i)$ be the projection of $A_i$ onto the $x$-axis and $\pi_y(A_i)$ the projection onto the $y$-axis.  As each line has slope $1$ or $-1$, the measures of the projections satisfy $\lambda(\pi_x(A_i)) = \lambda(\pi_y(A_i)).$  For each $i=1,2,3,4$, define $\mu^{(z)}_i(A) := \frac{1}{2} \lambda( \pi_x( A_i ) ) = \frac12 \lambda( \pi_y(A_i)).$
Finally we define the measure $\mu^{(z)} = \mu^{(z)}_1+\mu^{(z)}_2+\mu^{(z)}_3+\mu^{(z)}_4.$ It is simple to check from the construction above that the measure $\mu^{(z)}$ is indeed a permuton.

\begin{figure}[ht]
	\begin{center}		\includegraphics[scale=.8]{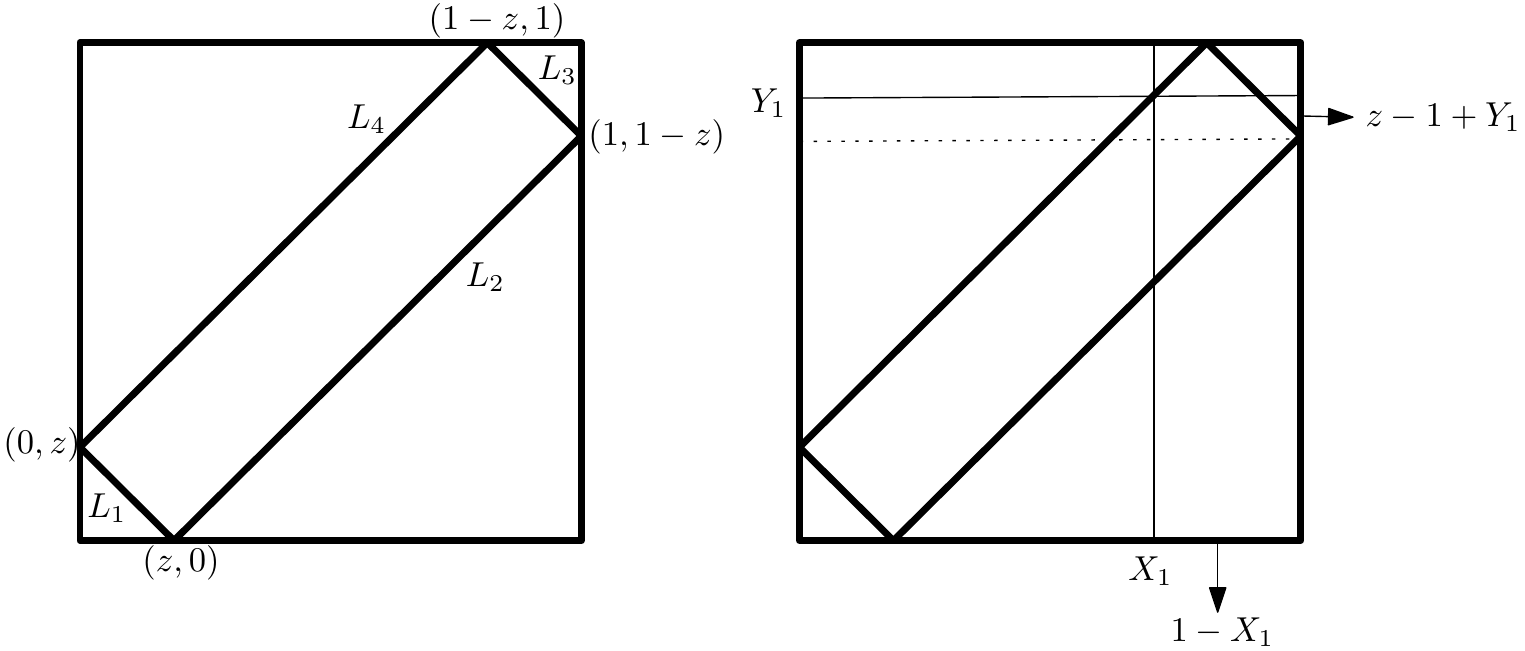}
		\caption{\textbf{Left:} A scheme for the construction of the permutons $\{\mu^{(z)}\}_{z\in[0,1]}$. 
			\textbf{Right:} A scheme for the proof of Proposition \ref{prop:CC_rect_perm}.}
		\label{fig:rect_perm}
	\end{center}
\end{figure}

\begin{proposition}\label{prop:CC_rect_perm}
	For all $z\in[0,1]$, the permuton $\mu^{(z)}$ is a \cc\ permuton with respect to $\sigma=21$.
\end{proposition}

\begin{proof}
	We consider $\big((X_i,Y_i)\big)_{i=1,2}\stackrel{i.i.d.}{\sim}\mu^{(z)}$ and we condition on the event $\left\{(X_1,Y_1)\in L_4\right\}$ (see the right-hand side  of Figure \ref{fig:rect_perm}). 
	Note that a.s.
	\begin{equation}\label{eq:always-z}
		\P(\pi_{{\bf X},{\bf Y}}=21|(X_1,Y_1))
		=
		\mu^{(z)}\left([0,X_1]\times [Y_1,1]\cup[X_1,1]\times [0,Y_1]\right)=\frac{1-X_1+z-1+Y_1}{2}=z,
	\end{equation}
	where in the last equality we used that $Y_1=X_1+z$ since $(X_1,Y_1)\in L_4$.
	One can check that the same relation as in the last equation holds also when conditioning on $\left\{(X_1,Y_1)\in L_1\right\}$, $\left\{(X_1,Y_1)\in L_2\right\}$, and $\left\{(X_1,Y_1)\in L_3\right\}$. So the statement easily follows.
\end{proof}

\begin{remark}
	We note that some randomized versions of the permutons $\{\mu^{(z)}\}_{z\in[0,1]}$ were already investigated in \cite{borga2020square,borga2021almost} as permuton limits of square and almost square permutations.
\end{remark}

We now construct a more general family of \cc\ permutons with respect to inversions that uses the permutons $\{\mu^{(z)}\}_{z\in[0,1]}$ as building blocks. We denote by $\mathcal S^*$ the set of permutations whose diagrams have the following property: every point in the diagram has the same number of points to its top-left and bottom-right sides (see the left-hand side of Figure \ref{fig:rect_perm2} for an example). Given $\eta \in \mathcal S^*$ and $z\in[0,1]$, we then consider the permuton $\mu^*=\mu^*(\eta,z)$ constructed as follows:
we replace each dot in the diagram of $\eta$ with a rescaled version of the same rectangle $R^{(z)}$ so that it is exactly inscribed in the box originally containing the dot (see the middle picture in Figure \ref{fig:rect_perm2}). Then we spread Lebesgue measure on the support obtained by the union of boundaries of these rescaled rectangles (see the right-hand side of Figure \ref{fig:rect_perm2}).

\begin{remark}
	We note that here are infinitely many  permutations in $\mathcal S^*$. Indeed for every permutation $\eta\in\mathcal S^*$ it is possible to consider the substitution $\eta[\eta,\dots,\eta]$ (see for instance \cite[Definition 1.2.]{bassino2017universal} for a definition of the substitution operation for permutations) that is again a permutation in $\mathcal S^*$ of size $|\eta|^2$. We also note that $\mathcal S^*_1=\mathcal S_1$, $\mathcal S^*_2=\mathcal S_2$, there are 2 permutations of size 3 in  $\mathcal S^*_3$, that are $123$ and $321$, and there are 4 permutations of size 4 in  $\mathcal S^*_4$, that are $1234$,$1234$,$2143$ and $3412$.
\end{remark}

\begin{figure}[ht]
	\begin{center}		\includegraphics[scale=.9]{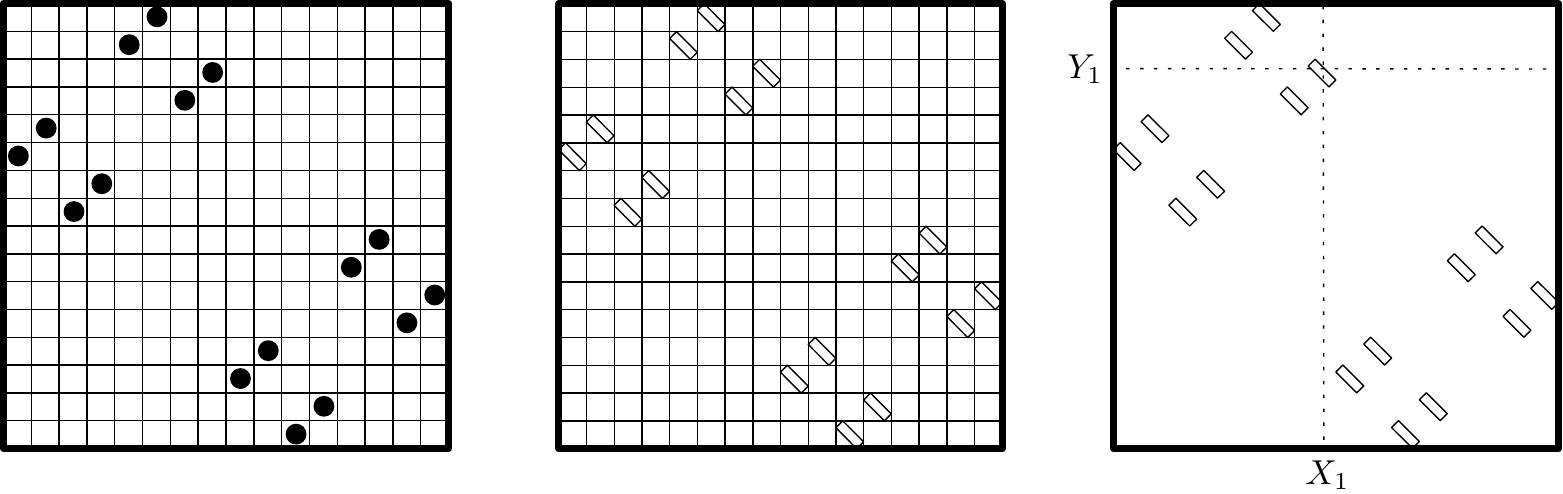}
		\caption{\textbf{Left:} The diagram of a permutation $\eta$ in $\mathcal S^*_{16}$. Note that each point in the diagram has exactly 10 points on its top-left and bottom-right sides. 
			\textbf{Middle:} The diagram of the permutation $\eta$ where each dot has been substituted by a rescaled rectangle $R^{(z)}$ for $z=0.8$.
			\textbf{Right:} The support of the permuton $\mu^*(\eta,z)$.}
		\label{fig:rect_perm2}
	\end{center}
\end{figure}

\begin{proposition}\label{prop:CC_rectangles_perm}
	For all $\eta \in \mathcal S^*$ and $z\in[0,1]$, the permuton $\mu^*=\mu^*(\eta,z)$ is a \cc\ permuton with respect to inversions.
\end{proposition}

\begin{proof}
	Let $\eta \in \mathcal S^*$ of size $n$ and assume that each point in the diagram of $\eta$ has exactly $k$ points on its top-left and bottom-right sides.
	Let $\big((X_i,Y_i)\big)_{i=1,2}\stackrel{i.i.d.}{\sim}\mu^{*}$. 
	Note that 
	\begin{equation}\label{eq:const_prob}
		\P(\pi_{{\bf X},{\bf Y}}=21|(X_1,Y_1))
		=
		\mu^{(z)}\left([0,X_1]\times [Y_1,1]\cup[X_1,1]\times [0,Y_1]\right)=\frac{k}{n}+\frac{z}{n}
	\end{equation}
	where in the last equality we used that there are exactly $k$ rectangles of mass $\frac 1 n$ in $[0,X_1]\times [Y_1,1]\cup[X_1,1]\times [0,Y_1]$ plus a portion or mass $\frac z n$ of the rectangle intersecting the point $(X_1,Y_1)$; for the latter claim we use exactly the argument of \eqref{eq:always-z}.
	Since the expression in \eqref{eq:const_prob} is independent of $(X_1,Y_1)$, the statement follows.
\end{proof}

It looks quite complicated to characterize all the possible \cc\ permutons with respect to inversions. For instance, note that with similar arguments as above, one can show that also the permuton in Figure \ref{ppn:counter} is a \cc\ permuton with respect to inversions and it is not one of the permutons considered in the previous proposition. 

\begin{figure}[ht]
	\begin{center}		\includegraphics[scale=.8]{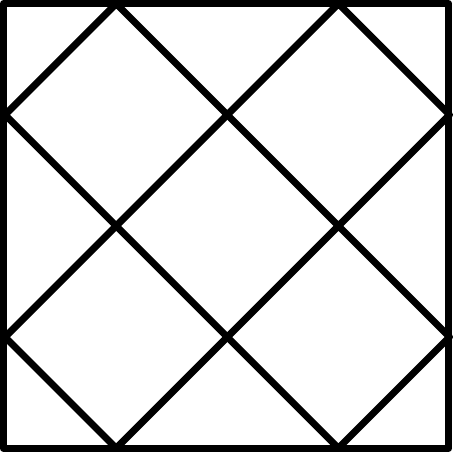}
		\caption{A \cc\ permuton with respect to inversions obtained by distributing Lebesgue measure on the boundaries of the four outer diamonds.}
		\label{fig:rect_perm3}
	\end{center}
\end{figure}

\appendix

\section{Large deviation results} \label{appA}

\begin{lemma}\label{lem:ldp}
	Let $(\mathcal{X},d), (\mathcal{Y},d')$ be metric spaces, equipped with Borel sigma fields $\mathcal{B}, \mathcal{B}'$ respectively. For every $n\ge 1$, let $X_n$ be a random variable taking values in $\mathcal{X}$. Assume that $(X_n)_{n\ge 1}$ satisfies an LDP with speed $n$ and good rate function $J(\cdot)$. Further assume that $T:\mathcal{X}\mapsto \mathcal{Y} $ be a function which is continuous on the set $$\{x\in \mathcal{X}:J(x)<\infty\}.$$ Then the following conclusions hold:
	\begin{enumerate}[(i)]
		\item The minimizers of $J(\cdot)$ over $\mathcal{X}$ are attained on a compact set $\mathcal F$, and $$d(X_n,\mathcal F)\xrightarrow{P}0.$$

		\item The sequence $T(X_n)$ satisfies an LDP on $\mathcal{Y}$ with speed $n$ and good rate function
		\begin{eqnarray*}
			\begin{split}
				J_1(x):=
				\begin{cases}
					\inf_{x\in\mathcal{X}:T(x)=t}J(x)&\text{ if }t\in T(\mathcal X),\\
					\infty& \text{ otherwise }.
				\end{cases}
			\end{split}
		\end{eqnarray*} 
		where the infimum over the empty set is taken to be $+\infty$ by convention.
		
		\bigskip
		
		\noindent We now further assume that $\mathcal{Y}=\R$ with $\mathcal{B}'=\mathcal{B}(\R)$ and that $T$ is bounded. Then we also have the following:
		
		\bigskip
		
		\item The following equation holds:
		$$\lim_{n\to\infty}\frac{1}{n}\log \E\Big[e^{nT(X_n)}\Big]=\sup_{x\in \mathcal{X}}\{T(x)-J(x)\}.$$
		
		\item Define a Gibbs measure $\mathcal{Q}_n$ on $\mathcal{X}$ by setting for all $B\in\mathcal{B}$,
		$$\mathcal{Q}_n(B):=\frac{\E\Big[ e^{nT(X_n)} \mathds{1}\{X_n\in B\}\Big]}{\E \Big[e^{nT(X_n)}\Big]}.$$
		If $Y_n\sim \mathcal{Q}_n$, then $Y_n$ satisfies an LDP with speed $n$ and good rate function $J_2$, given by $$J_2(x):=J(x)-T(x)-\inf_{y\in \mathcal{X}}\{J(y)-T(y)\}.$$
		
		\item Suppose $U\subset \R$ is open, such that $$\inf_{x\in \mathcal{X}:T(x)\in U}J(x)=\inf_{x\in \mathcal{X}: T(x)\in \bar{U}}J(x).$$
		Then, conditioned on the event $\{T(X_n)\in \bar{U}\}$, $X_n$ satisfies an LDP with speed $n$ and good rate function $J_3$, given by
		\begin{eqnarray*}
			\begin{split}
				J_3(x):=
				\begin{cases}
					J(x)-\inf_{y\in \mathcal{X}:T(y)\in \bar{U}}J(y)&\text{ if }T(x)\in \bar{U},\\
					\infty& \text{ otherwise }.
				\end{cases}
			\end{split}
		\end{eqnarray*}
	\end{enumerate}
\end{lemma}

\begin{proof}
	
	$(i)$ The fact that the minimizers of the good rate function are attained on a compact set $\mathcal F$ follows from standard analysis. To show convergence in probability, fixing $\varepsilon>0$, note that the set $\{x:d(x,\mathcal F)\ge \varepsilon\}$ is closed. Thus using the LDP, we get
	\begin{align*}
		\limsup_{n\to\infty}\frac{1}{n}\log\P(d(X_n,\mathcal F)\ge \varepsilon)\le -\inf_{x:d(x,\mathcal F)\ge \varepsilon}J(x).
	\end{align*}
	To show that $d(X_n,\mathcal F)\xrightarrow{P}0$, it suffices to show that $\inf_{x:d(x,\mathcal F)\ge \varepsilon}J(x)>0$. Suppose by way of contradiction we have $$\inf_{x:d(x,\mathcal F)\ge \varepsilon}J(x)=0.$$
	Since $J$ is a good rate function, the infimum on any closed set is attained at a point $x_0$, say, with $d(x_0,\mathcal F)\ge\varepsilon$. But $x_0$ is a global optimizer of $J(\cdot)$, a contradiction as $x_0\notin \mathcal F$.
	
	\medskip

	$(ii)$ Since $T$ is continuous on the set $\{x\in\mathcal{X}:J(x)<\infty\}$, the LDP for $T(X_n)$ follows on invoking the contraction principle (\cite[Thm 4.2.1]{DZ}). The theorem is stated for continuous functions, but the proof applies to functions which are continuous on the set $\{x\in\mathcal{X}:J(x)<\infty\}$ (see remark (c) following the theorem).
	
	\medskip
	
	$(iii)$ By part $(ii)$ above, the random variable $T(X_n)$ satisfies an LDP with the good rate function $J_1$. Since $T$ is also bounded, invoking Varadhan's Lemma, (\cite[Thm 4.3.1]{DZ}) we get
	\[\lim_{n\to\infty}\frac{1}{n}\log \E e^{nT(X_n)}= \sup_{t\in \R}\{t-J_1(t)\}.\]
	Since
	\[\sup_{x\in \mathcal{X}}\{T(x)-J(x)\}=\sup_{t\in \mathbb{R}} \sup_{x\in \mathcal{X}:T(x)=t} \{T(x)-J(x)\}=\sup_{t\in \mathbb{R}} \left\{t-\inf_{x\in \mathcal{X}:T(x)=t}J(x)\right\}=\sup_{t\in \R}\{t-J_1(t)\},\]
	the desired conclusion follows.
	
	\medskip
	
	$(iv)$ Let $R:\mathcal{X}\mapsto \R$ be a bounded continuous function. Then we have
	\begin{equation*}
		\E_{\mathcal{Q}_n} \left[ e^{n R(Y_n)}\right]=\frac{ \E e^{nR(X_n)+nT(X_n)}}{\E e^{nT(X_n)}}.
	\end{equation*}
	By part $(iii)$ applied to the functions $R+T$ and $T$ respectively, we have
	$$\lim_{n\to\infty}\frac{1}{n}\log  \E_{\mathcal{Q}_n} \left[ e^{nR(Y_n)}\right]=\sup_{x\in \mathcal{X}}\{R(x)+T(x)-J(x)\}-\sup_{x\in \mathcal{X}}\{T(x)-J(x)\}.$$
	By Bryc's inverse Varadhan Lemma (\cite[Thm 4.4.13]{DZ}), it follows that $Y_n$ satisfies the desired LDP. To invoke the theorem, one needs to check that $x\mapsto J(x)-T(x)$ is a good rate function, i.e., its level sets are compact. Fixing $\alpha<\infty$, set
	$A_\alpha:=\{x\in \mathcal{X}:J(x)-T(x)\le \alpha\}$, and let $(x_k)_{k\ge 1}$ be a sequence in $A_\alpha$. Note that
	$$A_\alpha\subseteq \{x\in \mathcal{X}:J(x)\le \|T\|_\infty+\alpha\}=:M_{\alpha}$$
	where $M_{\alpha}$ is compact since $J$ is a good rate function. Thus the sequence $(x_k)_{k\ge 1}$ has a subsequence which converges to $x_0\in M_{\alpha}$. To show compactness via sequential compactness, it suffices to show that $x_0\in A_\alpha$. But this follows on noting that
	$$\liminf_{k\to\infty}J(x_k)\ge J(x_0)\quad\text{ and }\quad \lim_{k\to\infty}T(x_k)=T(x_0).$$
	Here the first limit follows from lower semi-continuity of $J(\cdot)$, and the second limit follows from  the fact that $J(x_0)<\infty$, as $x_0\in M_{\alpha}$, along with the assumption that $T$ is continuous on the set $\{x\in \mathcal{X}:J(x)<\infty\}$.
	
	\medskip
	
	$(v)$ Let $A$ be any subset of $\mathcal{X}$ in $\mathcal{B}$. To show an LDP for the conditional measure, one needs to show that
	\begin{align}
		\label{eq:ldp_upper}\limsup_{n\to\infty}\frac{1}{n}\log  \P(X_n\in A|T(X_n)\in \bar{U})\le &-\inf_{x\in A, T(x)\in \bar{U}}J(x)+\inf_{x:T(x)\in \bar{U}}J(x),\text{ if $A$ is closed},\\
		\label{eq:ldp_lower}\liminf_{n\to\infty}\frac{1}{n}\log  \P(X_n\in A|T(X_n)\in \bar{U})\ge &-\inf_{x\in A, T(x)\in \bar{U}}J(x)+\inf_{x:T(x)\in \bar{U}}J(x),\text{ if $A$ is open}.
	\end{align}
	We show \eqref{eq:ldp_upper}, omitting the proof of \eqref{eq:ldp_lower}. To this effect, note that
	\begin{align*}
		\P(X_n\in A|T(X_n)\in \bar{U})=\frac{\P(X_n\in A, T(X_n)\in \bar{U})}{\P(T(X_n)\in \bar{U})}\le \frac{\P(X_n\in A, T(X_n)\in \bar{U})}{\P(T(X_n)\in U)}
	\end{align*}
	which on taking $\log$, dividing by $n$, and taking limits gives
	\begin{align}\label{eq:upper1}
		\limsup_{n\to\infty}\frac{1}{n}\log \P(X_n\in A|T(X_n)\in \bar{U})\le -\inf_{x\in \bar{B}}J(x)+\inf_{x\in C^\circ}J(x)
	\end{align}
	where $$B:=\{x\in A:T(x)\in \bar{U}\}, \quad C:=\{x\in \mathcal{X}:T(x)\in U\}.$$ 
	We now claim that if $A$ is closed, then we have
	\begin{align}\label{eq:upper2}
		\inf_{x\in \bar{B}}J(x)=\inf_{x\in B}J(x),\quad \inf_{x\in C^\circ}J(x)=\inf_{x\in C}J(x).
	\end{align}
	Given \eqref{eq:upper1} and \eqref{eq:upper2}, we get
	\begin{align*}
		\limsup_{n\to\infty}\frac{1}{n}\log \P(X_n\in A|T(X_n)\in \bar{U})\le -\inf_{x\in A:T(x)\in \bar{U}}J(x)+\inf_{x\in \mathcal{X}:T(x)\in U}J(x).
	\end{align*}
	From this \eqref{eq:ldp_upper} follows on using the assumption 
	$$\inf_{x\in \mathcal{X}:T(x)\in U}J(x)=\inf_{x\in \mathcal{X}:T(x)\in \bar{U}}J(x).$$ 
	It thus remains to verify \eqref{eq:upper2}, which we carry out below.
	
	\medskip
	
	\underline{Proof of $\inf_{x\in \bar{B}}J(x)=\inf_{x\in B}J(x)$.}  Since $B\subseteq \bar{B}$, we have $\inf_{x\in \bar{B}}J(x)\le \inf_{x\in B}J(x)$. Suppose the inequality is strict. Then we must have $\inf_{x\in \bar{B}}J(x)<\infty$.  Let $x_0\in \bar{B}/B$ be a point  where the infimum is attained (such an $x_0$ exists as $J(\cdot)$ is a good rate function and $\bar{B}$ is closed). Since $x_0\in \bar{B}$, there exists a sequence $(x_k)_{k\ge 1}\in B$ converging to $x_0$. But then $T(x_k)\in \bar{U}$ and $x_k\in A$, which, using the continuity of $T$ at $x_0$ and the closed-ness of $A$ and $\bar{U}$, implies that $T(x_0)\in \bar{U}$ and $x_0\in A$. This implies $x_0\in B$, a contradiction. This shows that $\inf_{x\in \bar{B}}J(x)=\inf_{x\in B}J(x)$, as desired.
	
	\medskip
	
	\underline{Proof of $\inf_{x\in C^\circ}J(x)=\inf_{x\in C}J(x)$:} Since $C^\circ\subseteq C$, we have $\inf_{x\in C}J(x)\le \inf_{x\in C^\circ}J(x)$. Suppose the inequality is strict. Then there exists $x_0\in C/C^\circ$ such that $J(x_0)<\inf_{x\in C^\circ}J(x)$. Since $x_0\notin C^\circ$, there exists a sequence $(x_k)_{k\ge 1}$ with $x_k\in C^c$ converging to $x_0$. But then we have $T(x_k)\in U^c$, which along with the continuity of $T$ at $x_0$ gives $T(x_0)\in U^c$. Thus $x_0\in C^c$, a contradiction. 
\end{proof}

\section{Other supporting lemmas} \label{appB}

\begin{lemma}\label{lem:dmat} Let $\nu\in \mathcal{M}$. 
	\begin{enumerate}[(i)]
		\item Set $A_x:=([0,x]\times [x,1])\cup ([x,1]\times [0,x])$ for $x\in [0,1]$, and 
		$\gamma:=\mathcal{O}(\nu)$. Then we have
		\begin{align}\label{eq:gamma}
			\gamma\left(A_{\nu(D_{11})}\right)\le 4\cdot \nu(D_{12}\cup D_{21}).
		\end{align}
		\item Suppose $\widetilde{\nu}$ is the push-forward of the measure $\nu$ under the map $T(x,y)=(1-x,1-y)$. If
		\begin{align}\label{eq:dmat}
			\nu(D_{11})-\nu(D_{22})>8\cdot\nu(D_{12}\cup D_{21}),
		\end{align} 
		we have $\mathcal{O}(\nu)\ne \mathcal{O}(\widetilde{\nu})$.
	\end{enumerate}
	
\end{lemma}

\begin{proof} 
	$(i)$ For convenience, set $\nu(D_{12})=a, \nu(D_{22})=b$, $\nu(D_{11})=c$ and $\nu(D_{21})=d$. 
	Observe that by definition of $\gamma=\mathcal{O}(\nu)$ (recall \eqref{eq:map_proj})
	\begin{align*}
		\gamma\left([0,c+a+d]^2\right) \ge \gamma\left([0,c+a]\times [0,c+d]\right) = c.
	\end{align*}
	Since $\gamma$ has uniform marginals,
	\begin{align*}
		\gamma\left([0,c]^2\right)\ge \gamma\left([0,c+a+d]^2\right)-2(a+d) \ge c-2(a+d).
	\end{align*}
	However, $\gamma([0,c]\times [0,1])=c$. Thus $\gamma([0,c]\times [c,1])\le 2(a+d)$. Similarly $\gamma([c,1]\times [0,c])\le 2(a+d)$. Thus, $\gamma(A_c)\le 4(a+d),$ as desired. 
	
	\medskip
	
	$(ii)$ Assume for the sake of contradiction that $\gamma=\mathcal{O}(\nu)= \mathcal{O}(\widetilde\nu)$. Then applying part $(a)$ we have
	\begin{align*}
		\gamma\left(A_{\nu(D_{11})}\cup A_{\nu(D_{22})}\right) \le 8\cdot \nu(D_{12}\cup D_{21}).
	\end{align*}
	But $\nu(D_{11})-\nu(D_{22})>0$, and so $$\gamma\left(A_{\nu(D_{11})}\cup A_{\nu(D_{22})}\right)\ge  \gamma\left([\nu(D_{11}),\nu(D_{22})]\times [0,1]\right)=\nu(D_{11})-\nu(D_{22}).$$
	Combining the last two displays contradict \eqref{eq:dmat}, and so the proof is complete.
\end{proof}

\bibliographystyle{alpha}		
\bibliography{pp}
\end{document}